\documentclass{article}

\usepackage{amsmath}
\usepackage[thmmarks, amsmath, amsthm]{ntheorem}
\usepackage{amssymb,mathrsfs}
\usepackage{hyperref,xcolor}
\usepackage{changepage}
\usepackage{enumitem} 
\usepackage{graphicx}
\usepackage{tikz}
\usepackage[utf8]{inputenc}
\usepackage[T1]{fontenc}
\usepackage{hyperref}
\usepackage{comment}

\allowdisplaybreaks

\numberwithin{equation}{section}

\newtheorem{myDefn}{Definition}[section]

\newtheorem{myProp}[myDefn]{Proposition}

\newtheorem{myRem}[myDefn]{Remark}

\newtheorem{myLem}[myDefn]{Lemma}

\newtheorem{myCor}[myDefn]{Corollary}

\newtheorem{myTheorem}[myDefn]{Theorem}

\DeclareMathOperator*{\argmin}{argmin}
\DeclareMathOperator*{\minn}{minimize}

\def\nn{\mathrm{n}}

\usepackage{indentfirst}

\def\R{\mathbb{R}}

\def\HH{\mathrm{H}}
\def\LL{\mathrm{L}}

\newcommand{\fonction}[5]{\begin{array}[t]{lrcl}#1 :&#2 &\longrightarrow &#3\\&#4& \longmapsto &#5 \end{array}}

\newcommand{\dual}[2]{\left\langle #1 , #2 \right\rangle}

\newlist{primenumerate}{enumerate}{1}
\setlist[primenumerate,1]{label={\roman*$'$}}

\title{Shape optimization involving the Tresca friction law in a 2D linear elastic model}

\author{Lo\"ic Bourdin\footnote{Institut de recherche XLIM. UMR CNRS 7252. Universit\'e de Limoges, France. \texttt{loic.bourdin@unilim.fr}}, 
Fabien Caubet\footnote{Universit\'e de Pau et des Pays de l'Adour, E2S UPPA, CNRS, LMAP, UMR 5142, 64000 Pau, France. \texttt{fabien.caubet@univ-pau.fr}}, Aymeric Jacob de Cordemoy\footnote{Sorbonne Université, Université Paris Cité, CNRS, INRIA, Laboratoire Jacques-Louis Lions, LJLL, F-75005 Paris, France.~\texttt{aymeric.jacob\_de\_cordemoy@sorbonne-universite.fr}
}}

\begin{document}

\maketitle

\begin{abstract}
The aim of this work is to analyse a shape optimization problem in a mechanical friction context. Precisely we perform a shape sensitivity analysis of a Tresca friction problem, that is, a boundary value problem involving the usual linear elasticity equations together with the (nonsmooth) Tresca friction law on a part of the boundary. We prove that the solution to the Tresca friction problem admits a directional shape derivative which moreover coincides with the solution to a boundary value problem involving tangential Signorini's unilateral conditions. Then an explicit expression of the shape gradient of the Tresca energy functional is provided (which allows us to provide numerical simulations illustrating our theoretical results). Our methodology is not based on any regularization procedure, but rather on the twice epi-differentiability of the (nonsmooth) Tresca friction functional which is analyzed thanks to a change of variables which is well-suited in the two-dimensional case. The obstruction in the higher-dimensional case is discussed.
\end{abstract}

\noindent \textbf{Keywords:} Shape optimization, shape sensitivity analysis, contact mechanics, Tresca's friction law, Signorini's unilateral conditions, variational inequalities, twice epi-differentiability.

\medskip

\noindent \textbf{AMS Classification:}
49Q10, 49Q12, 49J40, 74M10, 74M15, 74P10.

\section{Introduction}

Shape optimization problems involving (nonsmooth) mathematical models from contact mechanics (including for instance Signorini's unilateral conditions, Tresca's friction law, etc.) have already been investigated in the literature (see, e.g.,~\cite{BEREM,FULM,HASKLAR,HASLINGER,HEINEMANN,HINTERMULLERLAURAIN} and references therein). They can be treated by using, for example, regularization procedures (see~\cite{MAUALLJOU,CHAUDET,CHAUDET2}) or dualization procedures (see~\cite[Chapter~4]{SOKOZOL} and~\cite{SOKOZOLE2}).
In order to avoid distorting the physical meaning of the contact models or obtaining abstract results involving dual elements, we have introduced in a recent series of papers~\cite{4ABC,ABCJ,BCJDC,BCJDC2,jdc} a new methodology based on the notion of \textit{twice epi-differentiability} from the nonsmooth analysis literature (see, e.g.,~\cite{Rockafellar,ROCK}). In particular, this methodology has been successfully applied in~\cite{ABCJ} in order to analyze a shape optimization problem involving the Tresca friction law. Precisely, thanks to the twice epi-differentiability of the (nonsmooth) Tresca friction functional, we proved that the solution to the corresponding Tresca friction problem admits a directional shape derivative, which moreover coincides with the solution to a boundary value problem involving Signorini's unilateral conditions, and we provided an explicit expression of the shape gradient of the associated Tresca energy functional (which allowed us to provide numerical simulations illustrating our theoretical results).

However, as indicated in its title, the above paper~\cite{ABCJ} deals only with the \textit{scalar case} (which has no physical sense from the point of view of contact mechanics). Therefore the objective of the present paper is to discuss the applicability of our methodology to the \textit{elastic case} (which is the natural framework in contact mechanics). Here we would like to insist on the fact that this extension to the elastic case is not a simple replica of our previous paper~\cite{ABCJ}. Indeed, in addition to the obvious and significant difficulties inherent in calculations, our methodology leads in the elastic case to a major technical obstruction (that does not appear in the scalar case). The main contribution of the present paper is to show that a well-suited change of variables allows to overcome this obstruction in the two-dimensional elastic case. However, as discussed later in this introduction, the higher-dimensional elastic case remains an open challenge.

This long introduction is divided into several paragraphs in order to highlight (as concisely as possible) the major technical obstruction that appears in the application of our methodology in the general elastic case and how it can be overcome in the two-dimensional case.

\paragraph{Description of the shape optimization problem.}
In the sequel we will use standard notations, terminologies and assumptions that are precised in Section~\ref{secprelim}. Let~$d \geq 2$, $f\in\HH^{1}(\R^{d},\R^{d})$,~$g\in\HH^{2}(\R^{d},\R)$ such that $g>0$ \textit{a.e.}\ on~$\R^{d}$, and~$\Omega_{\mathrm{ref}}$ be a nonempty connected bounded open subset of~$\R^{d}$ with a~$\mathcal{C}^{1}$-boundary $\Gamma_{\mathrm{ref}}:= \mathrm{bd}({\Omega_{\mathrm{ref}}})$ (see Remark~\ref{regularityofn} for comments on this $\mathcal{C}^{1}$-regularity assumption) such that $\Gamma_{\mathrm{ref}}=\Gamma_{\mathrm{D}}\cup{\Gamma_{\mathrm{T}_{\mathrm{ref}}}}$, where $\Gamma_{\mathrm{D}}$ and ${\Gamma_{\mathrm{T}_{\mathrm{ref}}}}$ are two measurable (with positive measure) disjoint subsets of $\Gamma_{\mathrm{ref}}$. In this paper we consider the shape optimization problem with volume constraint given by
\begin{equation}\label{shapeOptim}
    \minn\limits_{ \substack{ \Omega\in \mathcal{U}_{\mathrm{ref}} \\ \vert \Omega \vert = \vert \Omega_{\mathrm{ref}} \vert } } \; \mathcal{J}(\Omega),
\end{equation}
where the set of admissible shapes is defined by
\begin{multline*}
     \mathcal{U}_{\mathrm{ref}} :=\biggl\{ \Omega\subset\R^{d} \mid \Omega \text{ nonempty connected bounded open subset of } \R^{d} \\[-12pt]  \text{ with a $\mathcal{C}^{1}$-boundary } \Gamma:= \mathrm{bd}({\Omega}) \text{ such that } \Gamma_{\mathrm{D}}\subset\Gamma  \biggl\},
\end{multline*}
where $\mathcal{J} : \mathcal{U}_{\mathrm{ref}} \to \R$ is the \textit{Tresca energy functional} defined by
\begin{equation*}\label{energyTresca}
    \mathcal{J}(\Omega) := \frac{1}{2}\int_{\Omega} \mathrm{A}\mathrm{e}\left(u_\Omega\right):\mathrm{e}\left(u_\Omega\right)+\int_{\Gamma_{\mathrm{T}}}g\left\|{u_\Omega}_\tau\right\|-\int_{\Omega}f\cdot u_{\Omega},
\end{equation*}
for all~$\Omega \in \mathcal{U}_{\mathrm{ref}}$, where~$u_\Omega \in\HH^{1}_{\mathrm{D}}(\Omega,\R^{d})$ stands for the unique 
weak solution to the \textit{Tresca friction problem} (see, e.g.,~\cite[Chapter~3 Section~5.2]{DUVAUTLIONS} or~\cite[Section~2.1.3]{BCJDC2}) given by
\begin{equation}\label{Trescaproblem2222}\tag{TP\ensuremath{{}_\Omega}}
\arraycolsep=2pt
\left\{
\begin{array}{rcll}
-\mathrm{div}(\mathrm{A}\mathrm{e}(u))-f & = & 0   & \text{ in } \Omega , \\
u & = & 0  & \text{ on } \Gamma_{\mathrm{D}} ,\\
\sigma_\nn(u) & = & 0  & \text{ on } \Gamma_{\mathrm{T}},\\
\left\|\sigma_\tau(u)\right\|\leq g \text{ and } u_{\tau}\cdot\sigma_{\tau}(u)+g\left\|u_{\tau}\right\| & = & 0  & \text{ on } \Gamma_{\mathrm{T}},
\end{array}
\right.
\end{equation}
where $\Gamma:=\mathrm{bd}(\Omega)$,  $\Gamma_{\mathrm{T}}:=\Gamma\textbackslash\Gamma_{\mathrm{D}}$ and
$$
\HH^{1}_{\mathrm{D}}(\Omega,\R^{d}) := \left\{v\in\HH^{1}(\Omega,\R^{d}) \mid v=0 \text{ \textit{a.e.}\ on }\Gamma_{\mathrm{D}} \right \}.
$$
The tangential boundary conditions on $\Gamma_{\mathrm{T}}$ in~\eqref{Trescaproblem2222} are known as the \textit{Tresca friction law}. Finally recall that the unique weak solution~$u_\Omega \in\HH^{1}_{\mathrm{D}}(\Omega,\R^{d})$ {to}~\eqref{Trescaproblem2222} is characterized by the variational inequality
\begin{equation*}
\displaystyle\int_{\Omega}\mathrm{A}\mathrm{e}(u_\Omega):\mathrm{e}(v-u_\Omega)+\int_{\Gamma_{\mathrm{T}}}g\left\|v_\tau\right\|-\int_{\Gamma_{\mathrm{T}}}g\left\|{u_\Omega}_\tau\right\| \geq\int_{\Omega}f\cdot\left(v-u_\Omega\right), \qquad \forall v\in\HH^{1}_{\mathrm{D}}(\Omega,\R^d),
\end{equation*}
and can be expressed as~$u_\Omega=\mathrm{prox}_{\phi_{\Omega}}(F_\Omega)$, where~$F_\Omega\in\HH^{1}_{\mathrm{D}}(\Omega,\R^{d})$
is the unique 
solution to the Dirichlet-Neumann problem (see, e.g.,~\cite[Section~2.1.1]{BCJDC2}) given by
\begin{equation}\label{besoinfortheend}
\left\{
\begin{array}{rcll}
-\mathrm{div}(\mathrm{A}\mathrm{e}(F))-f & = & 0   & \text{ in } \Omega , \\
F & = & 0  & \text{ on } \Gamma_{\mathrm{D}} ,\\
\mathrm{A}\mathrm{e}(F)\nn & = & 0  & \text{ on } \Gamma_{\mathrm{T}},
\end{array}
\right.
\end{equation} 
and $\mathrm{prox}_{\phi_{\Omega}} : \HH^1_{\mathrm{D}}(\Omega,\R^d) \to \HH^1_{\mathrm{D}}(\Omega,\R^d)$ stands for the \textit{proximal operator} (see Definition~\ref{proxi}) associated with the (convex) \textit{Tresca friction functional} $\phi_{\Omega} : \HH^1_{\mathrm{D}}(\Omega,\R^d) \to \R$ defined by
$$
\fonction{\phi_\Omega}{\HH^{1}_{\mathrm{D}}(\Omega,\R^d)}{\R}{v}{\displaystyle \int_{\Gamma_{\mathrm{T}}}g\left\|v_\tau\right\|.}
$$

\paragraph{Application of the classical strategy from (smooth) shape optimization literature.}
To deal with the numerical treatment of the above shape optimization problem, a suitable expression of the shape gradient of~$\mathcal{J}$ is required. For this purpose, we follow the classical strategy developed in (smooth) shape optimization literature (see, e.g.,~\cite{ALL,HENROT}). Consider~$\Omega_{0} \in \mathcal{U}_{\mathrm{ref}}$ and a direction~$\theta\in \mathcal{C}_{\mathrm{D}}^{2,\infty}(\R^{d},\R^{d})$ where
\begin{equation*}\label{direc12}
\mathcal{C}_{\mathrm{D}}^{2,\infty}(\R^{d},\R^{d}):=\left\{\theta\in\mathcal{C}^{2}(\R^{d},\R^{d})\cap\mathrm{W}^{2,\infty}(\R^{d},\R^{d}) \mid \theta=0 \text{ on } \Gamma_{\mathrm{D}}\right\}.
\end{equation*}
For any $t\geq0$ sufficiently small such that~$\mathrm{id}+t\theta$ is a~$\mathcal{C}^{2}$-diffeomorphism of $\R^{d}$, where~$\mathrm{id} :  \R^{d}\rightarrow \R^{d}$ stands for the identity map, we denote by~$\Omega_{t}:=(\mathrm{id}+t\theta)(\Omega_{0}) \in \mathcal{U}_{\mathrm{ref}}$ and by~$u_{t} := u_{\Omega_t} \in\HH^{1}_{\mathrm{D}}(\Omega_{t},\R^d)$ (note that~$u_{t}$ is defined on the moving domain~$\Omega_t$). To get an expression of the \textit{shape gradient} of~$\mathcal{J}$ at~$\Omega_0$ in the direction~$\theta$, defined by~$ \mathcal{J}'(\Omega_0)(\theta) := \lim_{t \to 0^+} \frac{ \mathcal{J} (\Omega_t) - \mathcal{J}(\Omega_0) }{t}$ (if it exists), the usual first step consists in introducing~$\overline{u}_{t}:=u_{t}\circ(\mathrm{id}+t\theta)\in\HH^{1}_{\mathrm{D}}(\Omega_{0},\R^d)$ (note that~$\overline{u}_{t}$ is defined on the fixed domain~$\Omega_0$) and obtaining an expression of the derivative (if it exists) of the map~$t \in \R_+ \mapsto \overline{u}_{t} \in\HH^{1}_{\mathrm{D}}(\Omega_0,\R^d)$ at~$t=0$, denoted by~$\overline{u}'_0 \in\HH^{1}_{\mathrm{D}}(\Omega_0,\R^d)$ and called  \textit{directional material derivative}. Then the \textit{directional shape derivative} is defined by~$u'_0 := \overline{u}'_0 - \nabla u_0 \theta $ which corresponds (roughly speaking) to the derivative of the map~$t \in \R_+ \mapsto u_{t} \in\HH^{1}_{\mathrm{D}}(\Omega_{t},\R^d)$ at~$t=0$. We refer to Remark~\ref{remdirectional} for a short discussion on the terminology \textit{directional} that has been added with respect to the classical literature on shape optimization.

To get an expression of the directional material derivative, we use the change of variables~$\mathrm{id}+t\theta$ and the equality
$$
\nn_t\circ(\mathrm{id}+t\theta)=\frac{(\mathrm{I}+t\nabla{\theta}^{\top})^{-1}\nn_0}{\left\|(\mathrm{I}+t\nabla{\theta}^{\top})^{-1}\nn_0\right\|},
$$ 
(see, e.g.,~\cite[Chapter 2 Proposition 2.48]{SOKOZOL}) in order to prove that~$\overline{u}_{t}\in\HH^{1}_{\mathrm{D}}(\Omega_{0},\R^d)$ is the unique solution to the parameterized variational inequality
\begin{multline}\label{inequalitytogetderivmat}
\displaystyle\int_{\Omega_0}\mathrm{J}_{t}\mathrm{A}\left[\nabla{\overline{u}_t}\left(\mathrm{I}+t\nabla{\theta}\right)^{-1}\right]:\nabla{\left(v-\overline{u}_t\right)}\left(\mathrm{I}+t\nabla{\theta}\right)^{-1}\\+\int_{\Gamma_{\mathrm{T}_0}}g_{t}\mathrm {J}_{\mathrm{T}_{t}}\left\|v-\left(v\cdot\frac{(\mathrm{I}+t\nabla{\theta}^{\top})^{-1}\nn_0}{\left\|(\mathrm{I}+t\nabla{\theta}^{\top})^{-1}\nn_0\right\|^2}\right)(\mathrm{I}+t\nabla{\theta}^{\top})^{-1}\nn_0\right\|\\-\int_{\Gamma_{\mathrm{T}_0}}g_{t}\mathrm {J}_{\mathrm{T}_{t}}\left\|\overline{u}_t-\left(\overline{u}_t\cdot\frac{(\mathrm{I}+t\nabla{\theta}^{\top})^{-1}\nn_0}{\left\|(\mathrm{I}+t\nabla{\theta}^{\top})^{-1}\nn_0\right\|^2}\right)(\mathrm{I}+t\nabla{\theta}^{\top})^{-1}\nn_0\right\|\\ \geq\int_{\Omega_0}f_{t}\mathrm{J}_{t}\cdot\left(v-\overline{u}_t\right), \qquad \forall v\in\HH^1_{\mathrm{D}}(\Omega_0,\R^d),
\end{multline}
where $f_{t}:=f\circ(\mathrm{id}+t\theta)\in\HH^{1}(\R^{d},\R^d)$, $g_t:=g\circ(\mathrm{id}+t\theta)\in\HH^{2}(\R^{d},\R)$,~$\mathrm{J}_{t}:=\mathrm{det}(\mathrm{I}+t\nabla{\theta})\in\LL^{\infty}(\R^{d},\R)$ is the Jacobian determinant,~$\mathrm {J}_{\mathrm{T}_{t}}:=\mathrm{det}(\mathrm{I}+t\nabla{\theta}) \|(\mathrm{I}+t\nabla{\theta}^{\top})^{-1}\nn_0 \| \in\mathcal{C}^{0}(\Gamma_{0},\R)$ is the tangential Jacobian and~$\mathrm{I}$ is the identity matrix of~$\R^{d \times d}$. Thus we get that
\begin{equation}\label{eqintro1}
    \displaystyle \overline{u}_t=\mathrm{prox}_{\overline{\phi}_{t}}(\overline{F}_{t}),
\end{equation}
where $\overline{F}_{t}\in\HH^1_{\mathrm{D}}(\Omega_0,\R^d)$ is the unique solution to the parameterized variational equality
\begin{equation*}
\displaystyle\int_{\Omega_0}\mathrm{J}_{t}\mathrm{A}\left[\nabla{\overline{F}_t}\left(\mathrm{I}+t\nabla{\theta}\right)^{-1}\right]:\nabla{v}\left(\mathrm{I}+t\nabla{\theta}\right)^{-1}=\int_{\Omega_0}f_{t}\mathrm{J}_{t}\cdot v, \qquad \forall v\in\HH^1_{\mathrm{D}}(\Omega_0,\R^d),
\end{equation*}
and $\mathrm{prox}_{\overline{\phi}_{t}} : \HH^1_{\mathrm{D}}(\Omega_0,\R^d) \to \HH^1_{\mathrm{D}}(\Omega_0,\R^d)$ is the proximal operator associated with the parameterized convex functional~$\overline{\phi}_{t} : \HH^1_{\mathrm{D}}(\Omega_0,\R^d) \to \R$ defined by
\begin{equation}\label{eqintro3}
\displaystyle\fonction{\overline{\phi}_{t}}{\HH^1_{\mathrm{D}}(\Omega_0,\R^d)}{\R}{v}{\displaystyle \int_{\Gamma_{\mathrm{T}_0}}g_{t}\mathrm {J}_{\mathrm{T}_{t}}\left\|v-\left(v\cdot\frac{(\mathrm{I}+t\nabla{\theta}^{\top})^{-1}\nn_0}{\left\|(\mathrm{I}+t\nabla{\theta}^{\top})^{-1}\nn_0\right\|^2}\right)(\mathrm{I}+t\nabla{\theta}^{\top})^{-1}\nn_0 \right\|.}
\end{equation}

\paragraph{Application of our methodology and facing a major obstruction.}
Now the difficulty (that does not appear in standard smooth shape optimization problems) is that, from~\eqref{eqintro1}, the differentiability of the map~$t \in \R_+ \mapsto \overline{u}_{t} \in\HH^{1}_{\mathrm{D}}(\Omega_{0},\R^d)$ at~$t=0$ is related to the differentiability (in a generalized sense) of the parameterized proximal operator~$\mathrm{prox}_{\overline{\phi}_{t}}$. For this purpose, the methodology that we have developed in~\cite{4ABC,ABCJ,BCJDC,BCJDC2,jdc} invokes the notion of \textit{twice epi-differentiability} for convex functions (introduced by Rockafellar in~\cite{Rockafellar}) which ensures the \textit{proto-differentiability} of the corresponding proximal operators. Actually, since the work by Rockafellar deals only with nonparameterized convex functions, we used instead the recent work~\cite{8AB} in which the notion of twice epi-differentiability has been extended to parameterized convex functions. The content of Proposition~\ref{TheoABC2018} in Appendix~\ref{appendix} (extracted from~\cite[Theorem 4.15]{8AB}) invites us to analyze the twice epi-differentiability of the parameterized convex function~$\overline{\phi}_t$ in order to obtain from~\eqref{eqintro1} the differentiability of the map~$t \in \R_+ \mapsto \overline{u}_{t} \in\HH^{1}_{\mathrm{D}}(\Omega_{0},\R^d)$ at~$t=0$ and a characterization of the directional material derivative~$\overline{u}'_0$.


However, at this step, a major technical obstruction appears in the elastic case (that does not appear in the scalar case in our previous paper~\cite{ABCJ}). Indeed the twice epi-differentiability of the parameterized convex functional~$\overline{\phi}_t$ is naturally related to the twice epi-differentiability of the parameterized convex integrand that appears in~\eqref{eqintro3}. As a reminder, in the scalar case~\cite{ABCJ}, the parameterized convex functional~$\overline{\phi}_t$ is given by the (simpler) expression
\begin{equation*}
\displaystyle\fonction{\overline{\phi}_{t}}{\HH^1_{\mathrm{D}}(\Omega_0,\R)}{\R}{v}{\displaystyle \int_{\Gamma_{\mathrm{T}_0}}g_{t}\mathrm {J}_{\mathrm{T}_{t}} \vert v \vert,} \qquad \text{[scalar case]}
\end{equation*}
and the twice epi-differentiability of the (simple) parameterized convex integrand can be analyzed. On the contrary, in the elastic case, the heavy expression of the parameterized convex integrand in~\eqref{eqintro3} does not allow, to our best knowledge, a tractable analysis of its twice epi-differentiability. 

At this step of our researches, we arrive to the conclusion that, even if our methodology based on the notion of twice epi-differentiability allows to analyze a shape optimization problem involving the Tresca friction law in the scalar case (see our previous paper~\cite{ABCJ}), we are not able, at least for now, to pursue our methodology in the general elastic case.


\paragraph{Overcoming the major obstruction in the two-dimensional case~$d=2$.}
In the two-dimensional case~$d=2$, one can fix~$\tau_0 \in \mathcal{C}^{0}(\Gamma_0,\R^2)$ an oriented (with an orientation arbitrarily fixed) orthonormal vector to~$\nn_0 \in \mathcal{C}^{0}(\Gamma_0,\R^2)$ and get that $(\mathrm{I}+t\nabla{\theta})\tau_0\cdot(\mathrm{I}+t\nabla{\theta}^{\top})^{-1} \nn_0=0$ on~$\Gamma_0$. Therefore Inequality~\eqref{inequalitytogetderivmat} can be rewritten as
\begin{multline*}\label{inequalityineedagain}
\displaystyle\int_{\Omega_0}\mathrm{J}_{t}\mathrm{A}\left[\nabla{\overline{u}_t}\left(\mathrm{I}+t\nabla{\theta}\right)^{-1}\right]:\nabla{\left(v-\overline{u}_t\right)}\left(\mathrm{I}+t\nabla{\theta}\right)^{-1}+\int_{\Gamma_{\mathrm{T}_0}}\frac{g_{t}\mathrm {J}_{\mathrm{T}_{t}}}{\left\|(\mathrm{I}+t\nabla{\theta})\tau_0\right\|}\left|v\cdot(\mathrm{I}+t\nabla{\theta})\tau_0\right|\\-\int_{\Gamma_{\mathrm{T}_0}}\frac{g_{t}\mathrm {J}_{\mathrm{T}_{t}}}{\left\|(\mathrm{I}+t\nabla{\theta})\tau_0\right\|}\left|\overline{u}_t\cdot(\mathrm{I}+t\nabla{\theta})\tau_0\right|\geq\int_{\Omega_0}f_{t}\mathrm{J}_{t}\cdot\left(v-\overline{u}_t\right), \qquad \forall v\in\HH^1_{\mathrm{D}}(\Omega_0,\R^2).
\end{multline*}
Therefore we introduce $\overline{\overline{u}}_t:=(\mathrm{I}+t\nabla{\theta}^{\top})\overline{u}_{t}\in\HH^1_{\mathrm{D}}(\Omega_0,\R^2)$ which satisfies
\begin{multline}\label{inequalityaagain23}
\displaystyle\int_{\Omega_0}\mathrm{J}_{t}\mathrm{A}\left[\nabla{\left(\left(\mathrm{I}+t\nabla{\theta}^{\top}\right)^{-1}\overline{\overline{u}}_t\right)}\left(\mathrm{I}+t\nabla{\theta}\right)^{-1}\right]:\nabla{\left(\left(\mathrm{I}+t\nabla{\theta}^{\top}\right)^{-1}\left(v-\overline{\overline{u}}_t\right)\right)}\left(\mathrm{I}+t\nabla{\theta}\right)^{-1}\\+\int_{\Gamma_{\mathrm{T}_0}}\frac{g_{t}\mathrm {J}_{\mathrm{T}_{t}}}{\left\|(\mathrm{I}+t\nabla{\theta})\tau_0\right\|} \vert v \cdot \tau_0 \vert -\int_{\Gamma_{\mathrm{T}_0}}\frac{g_{t}\mathrm {J}_{\mathrm{T}_{t}}}{\left\|(\mathrm{I}+t\nabla{\theta})\tau_0\right\|} \vert \overline{\overline{u}}_t \cdot \tau_0 \vert \\ \geq\int_{\Omega_0}\left(\mathrm{I}+t\nabla{\theta}\right)^{-1}f_{t}\mathrm{J}_{t}\cdot\left(v-\overline{\overline{u}}_t\right), \qquad \forall v\in\HH^1_{\mathrm{D}}(\Omega_0,\R^2),
\end{multline}
and thus can be expressed as~$\overline{\overline{u}}_t=\mathrm{prox}_{ \overline{\overline{\phi}}_t }(\overline{\overline{F}}_t)$ where $\overline{\overline{F}}_t\in\HH^{1}_{\mathrm{D}}(\Omega_0,\R^2)$ is the unique solution to the parameterized variational equality
\begin{multline}\label{equalitypastropcomp}
\int_{\Omega_0}\mathrm{J}_{t}\mathrm{A}\left[\nabla{\left(\left(\mathrm{I}+t\nabla{\theta}^{\top}\right)^{-1}\overline{\overline{F}}_t\right)}\left(\mathrm{I}+t\nabla{\theta}\right)^{-1}\right]:\nabla{\left(\left(\mathrm{I}+t\nabla{\theta}^{\top}\right)^{-1}v\right)}\left(\mathrm{I}+t\nabla{\theta}\right)^{-1}\\= \int_{\Omega_0}\left(\mathrm{I}+t\nabla{\theta}\right)^{-1} f_{t}\mathrm{J}_{t}\cdot v, \qquad \forall v\in\HH^{1}_{\mathrm{D}}(\Omega_0,\R^2),
\end{multline}
and~$\mathrm{prox}_{ \overline{\overline{\phi}}_t } : \HH^1_{\mathrm{D}}(\Omega_0,\R^2) \to \HH^1_{\mathrm{D}}(\Omega_0,\R^2)$ is the proximal operator associated with the parameterized convex functional $\overline{\overline{\phi}}_t : \HH^1_{\mathrm{D}}(\Omega_0,\R^2) \to \R$ defined by
\begin{equation}\label{trescaparajeneed2}
    \displaystyle\fonction{\overline{\overline{\phi}}_t}{\HH^1_{\mathrm{D}}(\Omega_0,\R^2)}{\R}{v}{\displaystyle\int_{\Gamma_{\mathrm{T}_0}}\frac{g_{t}\mathrm {J}_{\mathrm{T}_{t}}}{\left\|(\mathrm{I}+t\nabla{\theta})\tau_0\right\|} \vert v \cdot \tau_0 \vert.}
\end{equation} 
As we will see in the present paper (see Section~\ref{setting}), the parameterized convex integrand in the expression~\eqref{trescaparajeneed2} (simpler than the one in~\eqref{eqintro3}) allows a tractable analysis of its twice epi-differentiability, and therefore allows to continue our methodology (but only in the two-dimensional case~$d=2$). Precisely, thanks to the twice epi-differentiability of the parameterized convex functional $\overline{\overline{\phi}}_t$, we are able to obtain from Proposition~\ref{TheoABC2018} a characterization of the derivative of the map~$t \in \R_+ \mapsto \overline{\overline{u}}_{t} \in \HH^{1}_{\mathrm{D}}(\Omega_{0},\R^d)$ at~$t=0$, denoted by~$\overline{\overline{u}}'_0 \in \HH^{1}_{\mathrm{D}}(\Omega_{0},\R^d) $, and then to deduce successively a characterization of the directional material derivative given by~$\overline{u}'_{0}=\overline{\overline{u}}'_{0}-\nabla{\theta}^{\top}u_0\in\HH^{1}_{\mathrm{D}}(\Omega_0,\R^2)$, then a characterization of the directional shape derivative given by~$u'_0 := \overline{u}'_0 - \nabla u_0 \theta $, and finally an expression of the shape gradient~$\mathcal{J}'(\Omega_0)(\theta)$. These results are summarized in the next paragraph.

\begin{myRem}\label{remarquesurscalarprod}\normalfont
As described above, the changes of variables used in this paper lead to~$\overline{\overline{u}}_t=\mathrm{prox}_{ \overline{\overline{\phi}}_t }(\overline{\overline{F}}_t)$ where the  proximal operator is defined on the Hilbert space $\HH^1(\Omega_0,\R^2)$ endowed with the \textit{parameterized} scalar product given by
\begin{multline*}
    (v_1,v_2)\in\left(\HH^1_{\mathrm{D}}(\Omega_0,\R^2)\right)^2\longmapsto\\ \int_{\Omega_0}\mathrm{J}_{t}\mathrm{A}\left[\nabla{\left(\left(\mathrm{I}+t\nabla{\theta}^{\top}\right)^{-1}v_1\right)}\left(\mathrm{I}+t\nabla{\theta}\right)^{-1}\right]:\nabla{\left(\left(\mathrm{I}+t\nabla{\theta}^{\top}\right)^{-1}v_2\right)}\left(\mathrm{I}+t\nabla{\theta}\right)^{-1}\in\R,
\end{multline*}
 thus Proposition~\ref{TheoABC2018} cannot be applied. This difficulty can be overcome by adding the $t$-independent scalar product $\dual{\cdot}{\cdot}_{\HH^1(\Omega_0,\R^2)}$ (see~\eqref{scalarusuel11}) to both members of Inequality~\eqref{inequalityaagain23} which leads us to replace $\overline{\overline{F}}_t$ by another solution that satisfies a more complex variational equality than Equality~\eqref{equalitypastropcomp}. Actually, this difficulty also appears in the three-dimensional case $d=3$ which can be overcomed in the same manner, and also (in an easier way) in the scalar case in our previous paper~\cite{ABCJ}.
\end{myRem}

\paragraph{Main results in the two-dimensional case~$d=2$.} We summarize here our main theoretical results (given in Theorems~\ref{materialderiv1} and~\ref{shapederivofJTresca1}). However we present the directional material and shape derivatives, and the shape gradient of $\mathcal{J}$, under some additional regularity assumptions, precisely in the framework of Corollaries~\ref{materialderiv2},~\ref{shapederiv1} and~\ref{shapederivofJ}, because their expressions are more elegant in that case. Furthermore, to ease the notations, we will use the notations~$\nn := \nn_0$ and~$\tau := \tau_0$.
\begin{enumerate}
  \item[(i)] Under some appropriate assumptions described in Corollary~\ref{materialderiv2}, the map
$t\in\R_{+} \mapsto \overline{u}_{t} \in \HH^{1}_{\mathrm{D}}(\Omega_0,\R^2)$ is differentiable at $t=0$, and the directional material derivative $\overline{u}'_{0}\in\HH^{1}_{\mathrm{D}}(\Omega_0,\R^2)$ is the unique weak solution to the \textit{tangential Signorini problem} (see, e.g.,~\cite[Section~2.1.2]{BCJDC2}) given by
\begin{equation*}
{
{\arraycolsep=2pt
\left\{
\begin{array}{rcll}
-\mathrm{div}\left(\mathrm{A}\mathrm{e}(\overline{u}'_0)\right)+\mathrm{div}\left(\mathrm{A}\mathrm{e}\left(\nabla{u_0}\theta\right)\right) &=& 0  & \text{ in } \Omega_0 , \\[5pt]
\overline{u}'_0 & = & 0  & \text{ on } \Gamma_{\mathrm{D}} ,\\[5pt]
\sigma_{\nn}(\overline{u}'_0)-{\xi^{m}(\theta)}_\nn & = & 0  & \text{ on } \Gamma_{\mathrm{T_0}} ,\\[5pt]
\sigma_\tau (\overline{u}'_0)+p(\theta)\frac{u_{0_\tau}}{\left\|u_{0_\tau}\right\|} -{\xi^{m}(\theta)}_{\tau} & = & 0  & \text{ on } \Gamma_{\mathrm{T_0}^{u_0,g}_{\mathrm{N}}} ,\\[5pt]
{\overline{u}'_0}_{\tau}+\left(\nabla{\theta}^{\top}u_0\right)_\tau & = & 0  & \text{ on } \Gamma_{\mathrm{T_0}^{u_0,g}_{\mathrm{D}}},\\[5pt]
\left({\overline{u}'_0}_{\tau}+\left(\nabla{\theta}^{\top}u_0\right)_\tau\right)\in\R_{-}\frac{\sigma_{\tau}(u_0)}{g} \\ [5pt] 
 \left(\sigma_\tau (\overline{u}'_0)-p(\theta) \frac{\sigma_{\tau}(u_0)}{g}-{\xi^{m}(\theta)}_{\tau}\right)\cdot \frac{\sigma_{\tau}(u_0)}{g}\leq0 \\ [5pt]
 \left({\overline{u}'_0}_{\tau}+\left(\nabla{\theta}^{\top}u_0\right)_\tau\right)\cdot\left(\sigma_\tau (\overline{u}'_0)-p(\theta) \frac{\sigma_{\tau}(u_0)}{g}-{\xi^{m}(\theta)}_{\tau}\right)  & = & 0  & \text{ on } \Gamma_{\mathrm{T_0}^{u_0,g}_{\mathrm{S}}}.
\end{array}
\right.}}
\end{equation*}
where $\xi^m(\theta):=\left((\mathrm{A}\mathrm{e}\left({u_0}\right))\nabla{\theta}^{\top}+\mathrm{A}(\nabla{u_0}\nabla{\theta})+(\nabla{\theta}-\mathrm{div}(\theta)\mathrm{I})\mathrm{A}\mathrm{e}\left({u_0}\right)\right)\nn\in\LL^2(\Gamma_{\mathrm{T}_0},\R^2)$ and $p(\theta) :=\nabla{g}\cdot\theta+g\left( \mathrm{div}_{\tau}(\theta)-\nabla{\theta}\tau\cdot\tau\right)\in\LL^2(\Gamma_{\mathrm{T}_0})$, and where $\Gamma_{\mathrm{T_0}}$ is decomposed, up to a null set, as~$\Gamma_{\mathrm{T_0}^{u_0,g}_{\mathrm{N}}}\cup\Gamma_{\mathrm{T_0}^{u_0,g}_{\mathrm{D}}}\cup\Gamma_{\mathrm{T_0}^{u_0,g}_{\mathrm{S}}}$ (see details in Theorem~\ref{materialderiv1}). We emphasize the notable fact that the boundary conditions which appear on $\Gamma_{\mathrm{T_0}^{u_0,g}_{\mathrm{S}}}$ are called \textit{tangential Signorini's unilateral conditions} because they are very close to the classical Signorini unilateral conditions (see, e.g.,~\cite{15SIG,16SIG}) except that, here, they are concerned with the \textit{tangential} components (instead of the \textit{normal} components in the classical case).
\item[(ii)]  We deduce in Corollary~\ref{shapederiv1} that, under appropriate assumptions, \textit{the directional shape derivative}, defined by~$u'_{0}:=\overline{u}'_0-\nabla{u_0}\theta\in\HH^{1}_{\mathrm{D}}(\Omega_0,\R^2)$, is the unique weak solution to the tangential Signorini problem given by
\begin{equation*}
{
{\arraycolsep=2pt
\left\{
\begin{array}{rcll}
-\mathrm{div}\left(\mathrm{A}\mathrm{e}(u'_0)\right) &=& 0  & \text{ in } \Omega_0 , \\[5pt]
u'_0 & = & 0  & \text{ on } \Gamma_{\mathrm{D}} ,\\[5pt]
\sigma_{\nn}(u'_0)-{\xi^{s}(\theta)}_\nn & = & 0  & \text{ on } \Gamma_{\mathrm{T_0}} ,\\[5pt]
\sigma_\tau (u'_0) +p(\theta)\frac{u_{0_\tau}}{\left\|u_{0_\tau}\right\|} -{\xi^{s}(\theta)}_{\tau} & = & 0  & \text{ on } \Gamma_{\mathrm{T_0}^{u_0,g}_{\mathrm{N}}} ,\\[5pt]
{u'_0}_{\tau}- W(\theta)_\tau  & = & 0 & \text{ on } \Gamma_{\mathrm{T_0}^{u_0,g}_{\mathrm{D}}},\\[5pt]
\left({u'_0}_{\tau}-W(\theta)_\tau\right)\in\R_{-}\frac{\sigma_{\tau}(u_0)}{g} \\ [5pt] \text{and } \left(\sigma_\tau (\overline{u}'_0)-p(\theta) \frac{\sigma_{\tau}(u_0)}{g}-{\xi^{s}(\theta)}_{\tau}\right)\cdot \frac{\sigma_{\tau}(u_0)}{g}\leq0 \\ [5pt]
\text{and } \left({u'_0}_{\tau}-W(\theta)_\tau\right)\cdot\left(\sigma_\tau (\overline{u}'_0)-p(\theta) \frac{\sigma_{\tau}(u_0)}{g}-{\xi^{s}(\theta)}_{\tau}\right)  & = & 0  & \text{ on } \Gamma_{\mathrm{T_0}^{u_0,g}_{\mathrm{S}}}.
\end{array}
\right.}}
\end{equation*}
where $W(\theta):=-\nabla{\theta}^{\top}u_0-\nabla{u_0}\theta\in\HH^{1}(\Omega_0,\R^2)$ and
\begin{multline*}
    \xi^s(\theta):=\theta \cdot \nn\left(\partial_{\nn}\left(\mathrm{A}\mathrm{e}\left({u_0}\right)\nn\right)-\partial_{\nn}\left(\mathrm{A}\mathrm{e}\left({u_0}\right)\right)\nn\right)+\mathrm{A}\mathrm{e}\left({u_0}\right)\nabla_{\tau}\left(\theta\cdot\nn\right)-\nabla{\left(\mathrm{A}\mathrm{e}\left({u_0}\right)\nn\right)}\theta\\+\left(\nabla{\theta}-\mathrm{div}_{\tau}(\theta)\mathrm{I}\right)\mathrm{A}\mathrm{e}\left({u_0}\right)\nn\in\LL^2(\Gamma_{\mathrm{T}_0},\R^2),
    \end{multline*}
\item[(iii)] Finally the two previous items are used to obtain Corollary~\ref{shapederivofJ} asserting that, under appropriate assumptions, the shape gradient of $\mathcal{J}$ at~$\Omega_{0}$ in the direction~$\theta$ is given by
\begin{multline*}
    \mathcal{J}'(\Omega_{0})(\theta)=\int_{\Gamma_{\mathrm{T}_{0}}} \theta\cdot\nn\left(\frac{\mathrm{A}\mathrm{e}(u_0):\mathrm{e}(u_0)}{2}-f\cdot u_{0}-\sigma_{\tau}(u_0)\cdot\partial_{\nn}(u_0)+\left\|{u_0}_\tau \right\|\left(Hg+\partial_{\nn}g\right)\right)\\+\int_{\Gamma_{\mathrm{T_0}}}{u_0}_\nn\sigma_\tau (u_0) \cdot\tau\left(\nabla{\tau}\theta_{\tau}-\nabla{\theta}\tau\right)\cdot\nn, 
\end{multline*} 
where $H$ stands for the mean curvature of $\Gamma_{0}$. One can notice that $\mathcal{J}'(\Omega_0)$ depends only on~$u_0$ (and not on $u’_0$). Hence its expression is explicit and linear with respect to the direction~$\theta$ and allows us to exhibit a descent direction for $\mathcal{J}$ at~$\Omega_0$ (see Section~\ref{numericalsim} for details). Then, using this descent direction together with a basic Uzawa algorithm to take into account the volume constraint, we perform in Section~\ref{numericalsim} numerical simulations to solve the shape optimization problem~\eqref{shapeOptim} on a toy example.
\end{enumerate}

\paragraph{Obstruction in the higher-dimensional case~$d \geq 3$ and additional comments.} 

In the higher-dimensional case~$d \geq 3$, the parameterized convex functional is given by~\eqref{eqintro3}
and, to the best of our knowledge, we did not find any change of variables in order to simplify the expression of its integrand. Therefore, for almost all $s\in\Gamma_{\mathrm{T}_0}$, we would have to investigate the twice epi-differentiability of the parameterized convex map
$$ x\in\R^d \longmapsto g_{t}\mathrm {J}_{\mathrm{T}_{t}} \left\| x- \left( x\cdot\frac{(\mathrm{I}+t\nabla{\theta}^{\top})^{-1}\nn}{\left\|(\mathrm{I}+t\nabla{\theta}^{\top})^{-1}\nn\right\|^2} \right)(\mathrm{I}+t\nabla{\theta}^{\top})^{-1}\nn \right\|  \in\R_+, $$ 
that we did not succeed to prove and to compute. This is an highly nontrivial work and an interesting topic for further researches.

\begin{myRem}\normalfont
During our bibliographical researches, we discovered a surprising hypothesis made in the paper~\cite{SOKOZOLE2} that is concerned, as the present work, with a shape optimization problem in a two dimensional case, involving the Tresca friction law in a linear elastic model, but with a different methodology based on dualization. With our notations and framework, this hypothesis consists in assuming that, for sufficiently small~$t \geq 0$, it holds that
$$ (\mathrm{I}+t\nabla{\theta}) \nn_0 =  \frac{(\mathrm{I}+t\nabla{\theta}^{\top})^{-1}\nn_0}{\left\|(\mathrm{I}+t\nabla{\theta}^{\top})^{-1}\nn_0\right\|^2}. $$
With this hypothesis, we have observed that our methodology can be continued (even in the elastic case) by applying a well-suited change of variables in the expression of~$\overline{\phi}_t$ in order to get the simpler expression
$v\in\HH^1_{\mathrm{D}}(\Omega_0,\R^2) \rightarrow \int_{\Gamma_{\mathrm{T}_0}}g_{t}\mathrm {J}_{\mathrm{T}_{t}}\|(\mathrm{I}+t\nabla{\theta})\tau_0\| | v \cdot \tau_0 |,$
that allows a tractable analysis of the twice epi-differentiability of its integrand but which differs from~\eqref{trescaparajeneed2}. In this regard, we refer to our paper~\cite{jdc} (concerned with a shape optimization problem involving Signorini's unilateral conditions in the elastic case) in which the same change of variables has been applied (but without assuming the above hypothesis which is useless in the context of~\cite{jdc}). However we emphasize that the above hypothesis is not satisfactory. Indeed, one can easily construct numerous counterexamples for which the above equality is not satisfied. For example, in the two-dimensional case~$d=2$, one can easily construct a situation where
$$ \nabla \theta = \left( \begin{array}{cc}
   1  & 1 \\
-1     & 0
\end{array} \right) \quad \text{and} \quad \nn_0 = \left( \begin{array}{c}
   1   \\
 0
\end{array} \right), $$
for which the above equality is not true. We refer to~\cite[Remark A.2.2. p.198]{MAURY} for additional comments on this hypothesis and what its implies on the direction $\theta$.
\end{myRem}

\begin{myRem}\normalfont
We mention that our methodology has already been successfully applied in the elastic case (in any dimension) in our previous paper~\cite{BCJDC2} (with also the emergence of tangential Signorini's unilateral conditions), but in order to solve an optimal control problem. To be clear, we underline that the obstruction encountered in the present paper does not appear in the (simpler) context of~\cite{BCJDC2}.
\end{myRem}

\paragraph{Organization of the paper.} The paper is organized as follows. Section~\ref{mainresult} is
the core of the present work, where the main results are stated and proved. In Section~\ref{numericalsim}, numerical simulations are performed to solve the shape optimization problem~\eqref{shapeOptim} on a toy example. Finally, Appendices~\ref{appendix} and~\ref{appendix2} are dedicated to recalls on the notion of twice epi-differentiability and on differential geometry respectively.

\section{Main results in the two-dimensional case~$d=2$}\label{mainresult}

The whole paper is now dedicated to the two-dimensional case~$d=2$. Furthermore, all along the present section, variational equalities and variational inequalities will be involved, as well as boundary value problems (Dirichlet-Neumann problem, tangential Signorini problem, Tresca friction problem). We refer to~\cite[Section~2.1]{BCJDC2} for notions of strong/weak solutions, existence/uniqueness results and proximal expressions of the solutions.

This section is organized as follows. In Section~\ref{secprelim}, we precise the notations, terminologies and assumptions used in our linear elastic model involving the Tresca friction law. Section~\ref{setting} is the technical part of the present paper. There, the notion of twice epi-differentiability and Proposition~\ref{TheoABC2018} are used in order to characterize the derivative~$\overline{\overline{u}}'_0$ as the unique solution to a variational inequality (see Proposition~\ref{deriv11}). From that result, we deduce in Section~\ref{section23a} a characterization of the directional material derivative~$\overline{u}'_0$ as the unique solution to a variational inequality (see Theorem~\ref{materialderiv1}). Then, under additional regularity assumptions, we characterize the directional material derivative~$\overline{u}'_0$ and the directional shape derivative~$u'_0$ as the unique weak solutions to tangential Signorini problems (see Corollaries~\ref{materialderiv2} and~\ref{shapederiv1}). Finally, in Section~\ref{energyfunctional}, we provide an expression of the shape gradient~$\mathcal{J}'(\Omega)(\theta)$ of the Tresca energy functional~$\mathcal{J}$ (see Theorem~\ref{shapederivofJTresca1} and Corollary~\ref{shapederivofJ}).

\subsection{Precisions on the notations, terminologies and assumptions used in our model}\label{secprelim}
Consider the shape optimization problem~\eqref{shapeOptim} in the two-dimensional case~$d=2$ and let us give some precisions on the notations, terminologies and assumptions used in our model. 

First, the notation $\cdot$ stands for the standard inner product on $\R^2$ and $||\cdot||$ for the corresponding Euclidean norm. We denote by $\mathrm{B}(0,1)$ the unit open ball of~$\R^2$ centered at $0$, with its boundary denoted by~$\mathrm{bd}(\mathrm{B}(0,1))$. Finally we denote by~$:$ the scalar product on~$\R^{2 \times 2}$ defined by $\mathrm{B}:\mathrm{C}=\sum_{i=1}^2\mathrm{B}_i\cdot\mathrm{C}_i$ for all $\mathrm{B}$, $\mathrm{C}\in\R^{2 \times 2}$, where~$\mathrm{B}_i\in\R^2$ (resp.\ $\mathrm{C}_i\in\R^2$) stands for the transpose of the~$i$-th line of~$\mathrm{B}$ (resp.\ $\mathrm{C}$) for all $i\in\{1,2\}$. 

Second, in the Tresca friction problem~\eqref{Trescaproblem2222} for some~$\Omega \in \mathcal{U}_{\rm ref}$, recall that~$\mathrm{A} \in \mathrm{L}^{\infty}(\Omega,\R^{2^4})$ stands for the stiffness tensor, assumed to be linear with constant coefficients (denoted by $a_{ijkl}$ for all $\left(i,j,k,l\right)\in\left\{1,2\right\}^{4}$), and~$\mathrm{e}$ is the infinitesimal strain tensor defined by~$\mathrm{e} : v\in\HH^1(\Omega,\R^2) \mapsto (\nabla{v}+\nabla{ v}^{\top})/2 \in \mathrm{L}^{2}(\Omega,\R^{2 \times 2})$. In this paper we assume that there exists a constant~$\alpha>0$ such that all coefficients of $\mathrm{A}$ and $e$ (denoted by $\epsilon_{ij}$ for all $\left(i,j\right)\in\left\{1,2\right\}^{2}$) satisfy
 \begin{equation*}
a_{ijkl}=a_{jikl}=a_{lkij} \; \text{ and } \;
\displaystyle\sum_{i=1}^2\sum_{j=1}^2\sum_{k=1}^2\sum_{l=1}^2a_{ijkl}\epsilon_{ij}(v_1)(x)\epsilon_{kl}(v_2)(x)\geq\alpha\sum_{i=1}^2\sum_{j=1}^2\epsilon_{ij}(v_1)(x)\epsilon_{ij}(v_2)(x),
\end{equation*}    
for all $v_1$, $v_2\in\HH^1(\Omega,\R^2)$ and for \textit{a.e.}\ $x\in\Omega$. From the symmetry assumption on $\mathrm{A}$, note that~$\mathrm{A}\mathrm{e}(v)=\mathrm{A}\nabla{v}$ for all~$v\in\HH^1(\Omega,\R^2)$. Moreover, since $ \Gamma_{\mathrm{D}} $ has a positive measure, it follows that 
\begin{equation}\label{scalarusuel11}
    \fonction{\dual{\cdot}{\cdot}_{\HH^{1}_{\mathrm{D}}(\Omega,\R^2)}}{ \left(\HH^1_{\mathrm{D}}(\Omega,\R^2)\right)^2 }{\R}{(v_1,v_2)}{\displaystyle\int_{\Omega}\mathrm{A}\mathrm{e}(v_1):\mathrm{e}(v_2),}
\end{equation}
is a scalar product on $\HH^{1}_{\mathrm{D}}(\Omega,\R^2)$ (see, e.g.,~\cite[Chapter~3]{DUVAUTLIONS}) and we denote by $\left\|\cdot\right\|_{\HH^{1}_{\mathrm{D}}(\Omega,\R^2)}$ the corresponding norm.

Finally, for any~$\Omega \in \mathcal{U}_{\rm ref}$, the notation~$\nn \in \mathcal{C}^{0}(\Gamma,\R^2)$ stands for the outward-pointing unit normal vector to~$\Gamma$ and, since we deal with the two-dimensional case~$d=2$, we can fix~$\tau \in \mathcal{C}^{0}(\Gamma,\R^2)$ an oriented (with an orientation arbitrarily fixed) orthonormal vector to~$\nn$. For any~$v\in\LL^{2}(\Gamma,\R^2)$, one has the decomposition $v=v_{\nn}\nn+v_\tau$, where $v_{\nn}:=v\cdot\nn\in\LL^{2}(\Gamma,\R)$ and~$v_{\tau}:=v-v_{\nn}\nn = ( v \cdot \tau ) \tau \in\LL^2(\Gamma,\R^2)$. In particular, if the stress vector $\mathrm{A}\mathrm{e}(v)\nn$ belongs to $\LL^2(\Gamma,\R^2)$ for some~$v\in\HH^1(\Omega,\R^2)$, then~$\mathrm{A}\mathrm{e}(v)\nn = \sigma_{\nn}(v)\nn + \sigma_\tau(v)  $, where~$\sigma_{\nn}(v) := \mathrm{A}\mathrm{e}(v)\nn \cdot \nn \in\LL^{2}(\Gamma,\R)$ is the normal stress and~$\sigma_{\tau}(v) := \mathrm{A}\mathrm{e}(v)\nn - \sigma_{\nn}(v)\nn = (\mathrm{A}\mathrm{e}(v)\nn \cdot\tau)\tau \in\LL^{2}(\Gamma,\R^2)$ is the shear stress.



\subsection{Twice epi-differentiability and the derivative~$\overline{\overline{u}}'_0$}\label{setting}

As in Introduction, let~$\Omega_{0} \in \mathcal{U}_{\mathrm{ref}}$ and~$\theta\in \mathcal{C}_{\mathrm{D}}^{2,\infty}(\R^2,\R^2)$ be fixed for the whole section. For any~$t\geq0$ sufficiently small, we denote by~$\Omega_{t}:=(\mathrm{id}+t\theta)(\Omega_{0}) \in \mathcal{U}_{\mathrm{ref}}$ and by~$u_{t} := u_{\Omega_t} \in\HH^{1}_{\mathrm{D}}(\Omega_{t},\R^2)$. Then we introduce~$\overline{u}_{t}:=u_{t}\circ(\mathrm{id}+t\theta)\in\HH^{1}_{\mathrm{D}}(\Omega_{0},\R^2)$ and~$\overline{\overline{u}}_t:=(\mathrm{I}+t\nabla{\theta}^{\top})\overline{u}_{t}\in\HH^1_{\mathrm{D}}(\Omega_0,\R^2)$. In the sequel, to ease the notations, we denote by~$\nn := \nn_0$ and~$\tau := \tau_0$.

In this section our objective is to get a characterization of the derivative~$\overline{\overline{u}}'_0 \in\HH^1_{\mathrm{D}}(\Omega_0,\R^2)$ (if it exists). For this purpose, our methodology relies on the equality~$ \overline{\overline{u}}_t=\mathrm{prox}_{\overline{\overline{\phi}}_t }(\overline{\overline{F}}_t)$ (established in Introduction) and on the application of Proposition~\ref{TheoABC2018}. However, as mentioned in Remark~\ref{remarquesurscalarprod}, the equality~$ \overline{\overline{u}}_t=\mathrm{prox}_{\overline{\overline{\phi}}_t }(\overline{\overline{F}}_t)$
holds true, but considered on the Hilbert space $\HH^{1}_{\mathrm{D}}(\Omega_0,\R^2)$ endowed with the \textit{parameterized} scalar product given by
\begin{multline*}
    (v_1,v_2)\in\left(\HH^1_{\mathrm{D}}(\Omega_0,\R^2)\right)^2\longmapsto\\ \int_{\Omega_0}\mathrm{J}_{t}\mathrm{A}\left[\nabla{\left(\left(\mathrm{I}+t\nabla{\theta}^{\top}\right)^{-1}v_1\right)}\left(\mathrm{I}+t\nabla{\theta}\right)^{-1}\right]:\nabla{\left(\left(\mathrm{I}+t\nabla{\theta}^{\top}\right)^{-1}v_2\right)}\left(\mathrm{I}+t\nabla{\theta}\right)^{-1}\in\R.
\end{multline*}
Therefore one cannot apply Proposition~\ref{TheoABC2018} directly. To overcome this difficulty, as in~\cite{ABCJ}, let us add~$\dual{\overline{\overline{u}}_t}{v-\overline{\overline{u}}_t}_{\HH^{1}_{\mathrm{D}}(\Omega_0,\R^{2})}$ to both members of~Inequality~\eqref{inequalityaagain23}. Then, by using the equality~$\mathrm{B}:\mathrm{C}\mathrm{D}=\mathrm{B}\mathrm{D}^{\top}:\mathrm{C}$ which is true for all~$\mathrm{B}$, $\mathrm{C}$, $\mathrm{D}\in\R^{2\times 2}$, we obtain that
\begin{multline*}
\displaystyle\dual{\overline{\overline{u}}_t}{v-\overline{\overline{u}}_t}_{\HH^{1}_{\mathrm{D}}(\Omega_0,\R^{2})}+\int_{\Gamma_{\mathrm{T}_0}}\frac{g_{t}\mathrm {J}_{\mathrm{T}_{t}}}{\left\|(\mathrm{I}+t\nabla{\theta})\tau\right\|} \vert v \cdot \tau \vert -\int_{\Gamma_{\mathrm{T}_0}}\frac{g_{t}\mathrm {J}_{\mathrm{T}_{t}}}{\left\|(\mathrm{I}+t\nabla{\theta})\tau\right\|} \vert \overline{\overline{u}}_t \cdot \tau \vert \\\geq-\int_{\Omega_0}\mathrm{J}_{t}\mathrm{A}\left[\nabla{\left(\left(\mathrm{I}+t\nabla{\theta}^{\top}\right)^{-1}\overline{\overline{u}}_t\right)}\left(\mathrm{I}+t\nabla{\theta}\right)^{-1}\right]\left(\mathrm{I}+t\nabla{\theta}^{\top}\right)^{-1}:\nabla{\left(\left(\mathrm{I}+t\nabla{\theta}^{\top}\right)^{-1}\left(v-\overline{\overline{u}}_t\right)\right)}\\ +\int_{\Omega_0}\left(\mathrm{I}+t\nabla{\theta}\right)^{-1}f_{t}\mathrm{J}_{t}\cdot(v-\overline{\overline{u}}_t)+\int_{\Omega_0}\mathrm{A}\mathrm{e}\left({\overline{\overline{u}}_t}\right):\mathrm{e}{\left(v-\overline{\overline{u}}_t\right)}, \qquad \forall v\in\HH^{1}_{\mathrm{D}}(\Omega_0,\R^{2}),
\end{multline*}
and thus we get the equality~$ \overline{\overline{u}}_t=\mathrm{prox}_{\overline{\overline{\phi}}_t }(E_t)$ where $E_{t}\in\HH^{1}_{\mathrm{D}}(\Omega_0,\R^2)$ is the unique solution to the parameterized variational equality
\begin{multline*}\label{perturprob}
\displaystyle\dual{E_t}{v}_{\HH^{1}_{\mathrm{D}}(\Omega_0,\R^{2})}=\int_{\Omega_0}\left(\mathrm{I}+t\nabla{\theta}\right)^{-1}f_{t}\mathrm{J}_{t}\cdot v
\\-\int_{\Omega_0}\mathrm{J}_{t}\mathrm{A}\left[\nabla{\left(\left(\mathrm{I}+t\nabla{\theta}^{\top}\right)^{-1}\overline{\overline{u}}_t\right)}\left(\mathrm{I}+t\nabla{\theta}\right)^{-1}\right]\left(\mathrm{I}+t\nabla{\theta}^{\top}\right)^{-1}:\nabla{\left(\left(\mathrm{I}+t\nabla{\theta}^{\top}\right)^{-1}v\right)}\\+\int_{\Omega_0}\mathrm{A}\mathrm{e}\left({\overline{\overline{u}}_t}\right):\mathrm{e}\left({v}\right), \qquad \forall v\in\HH^{1}_{\mathrm{D}}(\Omega_0,\R^2),
\end{multline*}
considered on the Hilbert space $\HH^{1}_{\mathrm{D}}(\Omega_0,\R^2)$ endowed with the \textit{nonparameterized} scalar product~$\dual{\cdot}{\cdot}_{\HH^{1}_{\mathrm{D}}(\Omega_0,\R^2)}$. 

Furthermore, to be in accordance with the notations of Proposition~\ref{TheoABC2018}, we introduce the parameterized convex functional~$\Phi : \R_+\times\HH^1_{\mathrm{D}}(\Omega_0,\R^2) \to \R$ defined by
\begin{equation}\label{trescaparajeneed}
    \displaystyle\fonction{\Phi}{\R_+\times\HH^1_{\mathrm{D}}(\Omega_0,\R^2)}{\R}{(t,v)}{ \Phi(t,v) := \overline{\overline{\phi}}_t(v) = \displaystyle\int_{\Gamma_{\mathrm{T}_0}}\frac{g_{t}\mathrm {J}_{\mathrm{T}_{t}}}{\left\|(\mathrm{I}+t\nabla{\theta})\tau\right\|} \Vert v_\tau \Vert. }
\end{equation}
Hence, from the equality~$ \overline{\overline{u}}_t=\mathrm{prox}_{ \Phi(t,\cdot) }(E_t)$ satisfied on the Hilbert space $\HH^{1}_{\mathrm{D}}(\Omega_0,\R^2)$ endowed with the nonparameterized scalar product~$\dual{\cdot}{\cdot}_{\HH^{1}_{\mathrm{D}}(\Omega_0,\R^2)}$, we are now in a satisfactory setting in order to apply Proposition~\ref{TheoABC2018} (if its assumptions are satisfied of course). The first step is to analyze the differentiability of the map $t\in\R_{+} \mapsto E_{t} \in \HH^{1}_{\mathrm{D}}(\Omega_0,\R^2)$ at~$t=0$. For this purpose, let us recall from~\cite{HENROT} that:
\begin{enumerate}[label={\rm (\roman*)}]
    \item the map $t\in\R_{+} \mapsto \mathrm{J}_{t}\in\LL^{\infty}(\R^{2})$ is differentiable at $t=0$ with derivative given by~$\mathrm{div}(\theta)$;\label{diff11}
    \item the map $t\in\R_{+} \mapsto \left(\mathrm{I}+t\nabla{\theta}\right)^{-1}\in\LL^{\infty}(\R^{2},\R^{2 \times 2})$ is differentiable at $t=0$ with derivative given by~$-\nabla{\theta}$;\label{diff12}
    \item the map $t\in\R_{+} \mapsto \left(\mathrm{I}+t\nabla{\theta}^{\top}\right)^{-1}\in\LL^{\infty}(\R^{2},\R^{2 \times 2})$ is differentiable at $t=0$ with derivative given by~$-\nabla{\theta}^{\top}$;\label{diff13}
    \item the map $t\in\R_{+} \mapsto \left(\mathrm{I}+t\nabla{\theta}\right)^{-1}f_{t}\mathrm{J}_{t}\in\LL^{2}(\R^{2},\R^2)$ is differentiable at $t=0$ with derivative given by~$\mathrm{div}\left(f\theta^{\top}\right)-\nabla{\theta}f$;\label{diff14}
    \item \label{item5} the map $t\in\R_{+} \mapsto \frac{g_{t}\mathrm {J}_{\mathrm{T}_{t}}}{\left\|(\mathrm{I}+t\nabla{\theta})\tau\right\|}\in\LL^{2}( \Gamma_{\mathrm{T}_0})$ is differentiable at $t=0$ with derivative given by~$p(\theta):=\nabla{g}\cdot\theta+g\left( \mathrm{div}_{\tau}(\theta)-\nabla{\theta}\tau\cdot\tau\right)$.\label{diff166} 
   
\end{enumerate}

\begin{myLem}\label{lem876}
The map $t\in\R_{+} \mapsto E_{t} \in \HH^{1}_{\mathrm{D}}(\Omega_0,\R^2)$ is differentiable at $t=0$ and its derivative, denoted by $E'_{0} \in \HH^{1}_{\mathrm{D}}(\Omega_0,\R^2)$, is the unique solution to the variational equality given by
\begin{multline}\label{Woderov}
    \dual{E'_{0}}{v}_{\HH^{1}_{\mathrm{D}}(\Omega_0,\R^{2})}=\int_{\Omega_{0}}\left(\mathrm{div}\left(f\theta^{\top}\right)-\nabla{\theta}f\right)\cdot v\\+\int_{\Omega_{0}}\left(\left(\mathrm{A}\mathrm{e}\left({u_0}\right)\right)\nabla{\theta}^{\top}+\mathrm{A}\left(\nabla{u_0}\nabla{\theta}\right)+\mathrm{A}\mathrm{e}\left(\nabla{\theta}^{\top}u_0\right)-\mathrm{div}(\theta)\mathrm{A}\mathrm{e}\left({u_0}\right)\right):\nabla{v}\\+\int_{\Omega_0}\mathrm{A}\mathrm{e}\left({u_0}\right):\mathrm{e}\left(\nabla{\theta}^{\top}v\right), \qquad\forall v\in\HH^{1}_{\mathrm{D}}(\Omega_0,\R^{2}).
\end{multline}
\end{myLem}

\begin{proof}
Using the Riesz representation theorem, we denote by $Z \in\HH^{1}_{\mathrm{D}}(\Omega_0,\R^2)$ the unique solution to the above variational inequality~\eqref{Woderov}. From linearity and using differentiability results~\ref{diff11},~\ref{diff12},~\ref{diff13},~\ref{diff14}, one gets
\begin{multline*}
    \left\|\frac{E_{t}-E_{0}}{t}-Z\right\|_{\HH^{1}_{\mathrm{D}}(\Omega_0,\R^{2})}\leq \\C(\Omega_0, \mathrm{A}, \theta) \biggl(\left\|\frac{\left(\mathrm{I}+t\nabla{\theta}\right)^{-1}f_{t}\mathrm{J}_{t}-f}{t}-\left(\mathrm{div}\left(f\theta^{\top}\right)-\nabla{\theta}f\right)\right\|_{\LL^2(\R^2,\R^2)}\\+\left\|\overline{\overline{u}}_t-u_0\right\|_{\HH^{1}_{\mathrm{D}}(\Omega_0,\R^{2})}+\frac{o(t)}{t} \left\|\overline{\overline{u}}_t\right\|_{\HH^{1}_{\mathrm{D}}(\Omega_0,\R^{2})}\biggl),
\end{multline*}
for all $t > 0$ sufficiently small, where $C(\Omega_0,\mathrm{A},\theta)>0$ is a constant which depends on~$\Omega_0$, $\mathrm{A}$ and $\theta$, and where $o$ stands for the standard Bachmann–Landau notation, with $\frac{|o(t)|}{t}\rightarrow 0$ when~$t\rightarrow0^{+}$.  Therefore, to conclude the proof, we only need to prove the continuity of the map~$t\in\R_{+} \mapsto \overline{\overline{u}}_t\in \HH^{1}_{\mathrm{D}}(\Omega_0,\R^{2})$ at $t=0$. For this purpose, take~$v=u_{0}$ in the variational formulation of~$\overline{\overline{u}}_t$ and~$v=\overline{\overline{u}}_t$ in the variational formulation of $u_{0}$ to get that
\begin{multline*}
    \left\|\overline{\overline{u}}_t-u_0\right\|_{\HH^{1}_{\mathrm{D}}(\Omega_0,\R^{2})} \leq \\ C(\Omega_0,\mathrm{A},\theta)\biggl(\left\|\left(\mathrm{I}+t\nabla{\theta}^{\top}\right)f_{t}\mathrm{J}_{t}-f\right\|_{\LL^2(\R^2,\R^2)}+\left\|\frac{g_{t}\mathrm {J}_{\mathrm{T}_{t}}}{\left\|(\mathrm{I}+t\nabla{\theta})\tau\right\|}-g\right\|_{\LL^{2}( \Gamma_{\mathrm{T}_0},\R)}\\+\left\|\overline{\overline{u}}_t\right\|_{\HH^{1}_{\mathrm{D}}(\Omega_0,\R^{2})}\left(t+o(t)\right)\biggl),
\end{multline*}
for all $t\geq0$ sufficiently small. Hence, to conclude the proof, we only need to prove that the map~$t\in\R_{+} \mapsto\left\|\overline{\overline{u}}_t\right\|_{\HH^{1}_{\mathrm{D}}(\Omega_0,\R^{2})} \in\R$ is bounded for $t\geq0$ sufficiently small. For this purpose, take~$v=0$ in the variational formulation of~$\overline{\overline{u}}_t$ to get that
\begin{multline*}
\left\|\overline{\overline{u}}_t\right\|^2_{\HH^{1}_{\mathrm{D}}(\Omega_0,\R^{2})} \leq \\
C(\Omega_0,\mathrm{A},\theta)\left(\left\|\left(\mathrm{I}+t\nabla{\theta}^{\top}\right)f_{t}\mathrm{J}_{t}\right\|_{\LL^2(\R^2,\R^2)}+\left\|\frac{g_{t}\mathrm {J}_{\mathrm{T}_{t}}}{\left\|(\mathrm{I}+t\nabla{\theta})\tau\right\|}\right\|_{\LL^{2}( \Gamma_{\mathrm{T}_0},\R)}\right)\left\|\overline{\overline{u}}_t\right\|_{\HH^{1}_{\mathrm{D}}(\Omega_0,\R^{2})} \\
+C(\Omega_0,\mathrm{A},\theta)\left\|\overline{\overline{u}}_t\right\|^2_{\HH^{1}_{\mathrm{D}}(\Omega_0,\R^{2})}\left(t+o(t)\right),
\end{multline*}
for all $t\geq0$ sufficiently small. Thus one deduces
$$
\left\|\overline{\overline{u}}_t\right\|_{\HH^{1}_{\mathrm{D}}(\Omega_0,\R^{2})}\leq \frac{C(\Omega_0,\mathrm{A},\theta)\left(\left\|\left(\mathrm{I}+t\nabla{\theta}^{\top}\right)f_{t}\mathrm{J}_{t}\right\|_{\LL^2(\R^2,\R^2)}+\left\|\frac{g_{t}\mathrm {J}_{\mathrm{T}_{t}}}{\left\|(\mathrm{I}+t\nabla{\theta})\tau\right\|}\right\|_{\LL^{2}( \Gamma_{\mathrm{T}_0})}\right)}{1-C(\Omega_0,\mathrm{A},\theta)\left(t+o(t)\right)},
$$
for all $t\geq0$ sufficiently small, and using the continuity of the map $t\in\R_{+} \mapsto (\mathrm{I}+t\nabla{\theta}^{\top})f_{t}\mathrm{J}_{t}\in\LL^{2}(\R^{2},\R^2)$ (see~\ref{diff14}) and of the map $t\in\R_{+} \mapsto \frac{g_{t}\mathrm {J}_{\mathrm{T}_{t}}}{\left\|(\mathrm{I}+t\nabla{\theta})\tau\right\|}\in\LL^{2}( \Gamma_{\mathrm{T}_0})$ (see~\ref{diff166}), the proof is complete.
\end{proof}

Now the second step is to investigate the twice epi-differentiability of the parameterized convex functional~$\Phi$ defined in~\eqref{trescaparajeneed}, as we did in our previous paper~\cite{BCJDC2} from which the next two lemmas are extracted. Precisely, to derive the next two lemmas, one has to apply~\cite[Propositions~2.18 and~2.23]{BCJDC2} on the particular case given by the expression~\eqref{trescaparajeneed}. 
For the needs of these lemmas, and to avoid any confusion, we recall that the notation~$\partial$ stands for the notion of \textit{subdifferential} (see Appendix~\ref{appendix}). We also introduce the notation~$x_{\tau(s)} := ( x \cdot \tau (s) ) \tau(s)$ for all~$x \in \R^2$ and all~$s \in \Gamma_0$. Similarly we will use, for all~$s\in\Gamma_0$, the \textit{tangential norm map} given by~$\left\|\cdot_{\tau(s)}\right\| : x \in\R^{2} \mapsto \left\| x_{\tau(s)} \right\| = \vert x \cdot \tau (s) \vert \in\R^+$. 

\begin{myLem}[Second-order difference quotient functions of $\Phi$]\label{epidiffoffunctionG}
For all $t>0$, $u\in \HH^{1}_{\mathrm{D}}(\Omega_0,\R^{2})$ and~$v\in\partial\Phi(0,\cdot)(u)$, it holds that
\begin{equation}\label{Delta2}
      \displaystyle\Delta_{t}^{2}\Phi(u\mid v)(\varphi)=\int_{\Gamma_{\mathrm{T}_0}}\Delta_{t}^{2}G(s)(u(s)\mid\sigma_{\tau}(v)(s))(\varphi(s)) \, \mathrm{d}s,
\end{equation}
for all $\varphi\in\HH^{1}_{\mathrm{D}}(\Omega_0,\R^{2})$, where, for almost all $s\in\Gamma_{\mathrm{T}_0}$, $\Delta_{t}^{2}G(s)(u(s)|\sigma_{\tau}(v)(s))$ stands for the second-order difference quotient functions of $G(s)$ at $u(s)\in\R^{2}$ for $\sigma_{\tau}(v)(s)\in \partial G(s)(0,\cdot)(u(s))=g(s)\partial ||\mathord{\cdot}_{\tau(s)}||(u(s))$, with~$G(s)$ defined by $$ 
\fonction{G(s)}{\mathbb{R}_{+}\times\mathbb{R}^{2}}{\R}{(t,x)}{G(s)(t,x):=\dfrac{g_{t}(s)\mathrm {J}_{\mathrm{T}_{t}}(s)}{\left\|(\mathrm{I}+t\nabla{\theta}(s))\tau(s)\right\|}\left\|x_{\tau(s)}\right\|.}
$$
\end{myLem}

\begin{myRem}\normalfont\label{regularityofn}
    The assumption that $\Gamma_0$ is of class $\mathcal{C}^1$ is made to ensure that $\nn\in \mathcal{C}^{0}(\Gamma,\R)$ which gives us, for all $s\in\Gamma_{\mathrm{T}_0}$ and all $x\in\R^2$, the continuity of the map $s\in\Gamma_{\mathrm{T}_0}\mapsto \left\|x_{\tau(s)}\right\|\in\mathbb{R}_{+}$. This property is used in the proof of~\cite[Proposition~2.18]{BCJDC2} (precisely, in the proof of~\cite[Lemma~2.16]{BCJDC2}).
\end{myRem}


\begin{myLem}[Second-order epi-derivative of $G(s)$]\label{épidiffgabs}
Assume that, for almost all $s\in\Gamma_{\mathrm{T}_0}$, $g$ has a directional derivative at $s$ in any direction. Then, for almost all $s\in\Gamma_{\mathrm{T}_0}$, the map~$G(s)$ is twice epi-differentiable at any~$x\in\mathbb{R}^{2}$ for all $y\in 
\partial G(s)(0,\cdot)(u(s))=g(s)\partial{||\mathord{\cdot}_{\tau(s)}||}(x)$ with
\begin{equation*}
\mathrm{D}_{e}^{2}G(s)(x \mid y)(z):=
\left\{
\begin{array}{lcll}
p(\theta)(s)\frac{x_{\tau(s)}}{\left\|x_{\tau(s)}\right\|}\cdot z	&   & \text{ if } x_{\tau(s)}\neq0 , \\
\iota_{\mathrm{N}_{\overline{\mathrm{B}(0,1)}\cap\left(\R\nn(s)\right)^{\perp}}(\frac{y}{g(s)})}(z)+p(\theta)(s)\frac{y}{g(s)}\cdot z	&   & \text{ if } x_{\tau(s)}=0,
\end{array}
\right.
\end{equation*}
for all $z\in\mathbb{R}^{2}$, where $p(\theta)\in\LL^{2}( \Gamma_{\mathrm{T}_0})$ is defined as the derivative at $t=0$ of the map $t\in\R_{+} \mapsto \frac{g_{t}\mathrm {J}_{\mathrm{T}_{t}}}{\left\|(\mathrm{I}+t\nabla{\theta})\tau\right\|}\in\LL^{2}( \Gamma_{\mathrm{T}_0})$ (see~\ref{item5}), $\mathrm{N}_{\overline{\mathrm{B}(0,1)}\cap\left(\R\nn(s)\right)^{\perp}}(\frac{y}{g(s)})$ is the normal cone to $\overline{\mathrm{B}(0,1)}\cap\left(\R\nn(s)\right)^{\perp}$ at~$\frac{y}{g(s)}$ given by \begin{equation*}
\mathrm{N}_{\overline{\mathrm{B}(0,1)}\cap\left(\R\nn(s)\right)^{\perp}}\left(\frac{y}{g(s)}\right)=
\left\{
\begin{array}{lcll}
\R\nn(s)	&   & \text{ if } \frac{y}{g(s)}\in\mathrm{B}(0,1)\cap\left(\R\nn(s)\right)^{\perp}, \\
\R\nn(s)+\R_{+}\frac{y}{g(s)}	&   & \text{ if } \frac{y}{g(s)}\in\mathrm{bd}{\left(\mathrm{B}(0,1)\right)}\cap\left(\R\nn(s)\right)^{\perp},
\end{array}
\right.
\end{equation*}
and $\iota_{\mathrm{N}_{\overline{\mathrm{B}(0,1)}\cap\left(\R\nn(s)\right)^{\perp}}(\frac{y}{g(s)})}$ stands for the indicator function of $\mathrm{N}_{\overline{\mathrm{B}(0,1)}\cap\left(\R\nn(s)\right)^{\perp}}(\frac{y}{g(s)})$ which is defined by~$\iota_{\mathrm{N}_{\overline{\mathrm{B}(0,1)}\cap\left(\R\nn(s)\right)^{\perp}}(\frac{y}{g(s)})}(z):=0$ if $z\in\mathrm{N}_{\overline{\mathrm{B}(0,1)}\cap\left(\R\nn(s)\right)^{\perp}}(\frac{y}{g(s)})$, and~$\iota_{\mathrm{N}_{\overline{\mathrm{B}(0,1)}\cap\left(\R\nn(s)\right)^{\perp}}(\frac{y}{g(s)})}(z):=+\infty$ otherwise.
\end{myLem}

We are now in a position to prove, under appropriate assumptions, that the map $t\in\R_{+} \mapsto \overline{\overline{u}}_t \in \HH^{1}_{\mathrm{D}}(\Omega_0,\R^2)$ is differentiable at $t=0$.

\begin{myProp}[The derivative~$\overline{\overline{u}}'_0$]\label{deriv11}
Assume that:
\begin{enumerate}[label={\rm{(H\arabic*)}}]
    \item\label{hypo1} for almost all $s\in\Gamma_{\mathrm{T}_0}$, $g$ has a directional derivative at $s$ in any direction.
    \item\label{hypo3} the parameterized convex functional $\Phi$ defined in~\eqref{trescaparajeneed} is twice epi-differentiable (see Definition~\ref{epidiffpara}) at $u_{0}$ for $E_{0}-u_{0}\in\partial \Phi(0,\cdot)(u_{0})$, with
\begin{equation}\label{hypoth1}
\displaystyle\mathrm{D}_{e}^{2}\Phi(u_{0}\mid E_{0}-u_{0})(\varphi)=\int_{\Gamma_{\mathrm{T}_0}}\mathrm{D}_{e}^{2}G(s)(u_{0}(s)\mid\sigma_{\tau}(E_{0}-u_{0})(s))(\varphi(s)) \, \mathrm{d}s,
\end{equation}    
for all $\varphi\in \HH^{1}_{\mathrm{D}}(\Omega_0,\R^2)$.
\end{enumerate}
Then the map
$t\in\R_{+} \mapsto \overline{\overline{u}}_t \in \HH^{1}_{\mathrm{D}}(\Omega_0,\R^2)$ is differentiable at $t=0$ and its derivative~$\overline{\overline{u}}'_0 \in \mathcal{K}_0$ is the unique solution to the variational inequality
\begin{multline*}\label{deriveavantmat}
\dual{\overline{\overline{u}}'_{0}}{\varphi-\overline{\overline{u}}'_{0}}_{\HH^{1}_{\mathrm{D}}(\Omega_0,\R^{2})}\geq-\int_{\Omega_0}\mathrm{div}\left(\mathrm{div}\left(\mathrm{A}\mathrm{e}(u_0)\right)\theta^{\top}\right)\cdot\left(\varphi-\overline{\overline{u}}'_0\right)\\+\int_{\Omega_0}\left(\left(\mathrm{A}\mathrm{e}\left({u_0}\right)\right)\nabla{\theta}^{\top}+\mathrm{A}\left(\nabla{u_0}\nabla{\theta}\right)+\mathrm{A}\mathrm{e}{\left(\nabla{\theta}^{\top}u_0\right)}-\mathrm{div}(\theta)\mathrm{A}\mathrm{e}\left({u_0}\right)\right):\nabla{\left(\varphi-\overline{\overline{u}}'_{0}\right)}\\+\dual{\mathrm{A}\mathrm{e}(u_0)\nn}{\nabla{\theta}^{\top}\left(\varphi-\overline{\overline{u}}'_{0}\right)}_{\HH^{-1/2}(\Gamma_0,\R^2)\times\HH^{1/2}(\Gamma_0,\R^2)}-\int_{\Gamma_{\mathrm{T_0}^{u_0,g}_{\mathrm{N}}}}p(\theta)\frac{u_{0_\tau}}{\left\|u_{0_\tau}\right\|}\cdot \left(\varphi_\tau-\overline{\overline{u}}'_{0_\tau}\right)\\+\int_{\Gamma_{\mathrm{T_0}^{u_0,g}_{\mathrm{S}}}}p(\theta)\frac{\sigma_{\tau}(u_0)}{g}\cdot \left(\varphi_\tau-\overline{\overline{u}}'_{0_\tau}\right), \qquad \forall \varphi \in \mathcal{K}_0,
\end{multline*}
where~$\mathcal{K}_0$ is the nonempty closed convex subset of $\HH^{1}_{\mathrm{D}}(\Omega_0,\R^2)$ given by
\begin{equation}\label{eqK0}
    \mathcal{K}_0 = \left\{ \varphi\in \HH^{1}_{\mathrm{D}}(\Omega_0,\R^2)\mid \varphi_\tau=0 \text{ \textit{a.e.}\ on } \Gamma_{\mathrm{T_0}^{u_0,g}_{\mathrm{D}}} \text{ and } \varphi_\tau\in\R_{-}\frac{\sigma_{\tau}(u_0)}{g} \text{ \textit{a.e.}\ on } \Gamma_{\mathrm{T_0}^{u_0,g}_{\mathrm{S}}} \right\},
\end{equation}
and where $\Gamma_{\mathrm{T_0}}$ is decomposed, up to a null set, as $\Gamma_{\mathrm{T_0}^{u_0,g}_{\mathrm{N}}}\cup\Gamma_{\mathrm{T_0}^{u_0,g}_{\mathrm{D}}}\cup\Gamma_{\mathrm{T_0}^{u_0,g}_{\mathrm{S}}}$ with
$$
\begin{array}{l}
\Gamma_{\mathrm{T_0}^{u_0,g}_{\mathrm{N}}}:=\left\{s\in\Gamma_{\mathrm{T_0}} \mid  u_{0_\tau}(s)\neq0\right \}, \\
\Gamma_{\mathrm{T_0}^{u_0,g}_{\mathrm{D}}}:=\left\{s\in\Gamma_{\mathrm{T_0}} \mid  u_{0_\tau}(s)=0 \text{ and } \frac{\sigma_{\tau}(u_0)(s)}{g(s)}\in\mathrm{B}(0,1)\cap\left(\R\nn(s)\right)^{\perp}\right\}, \\
\Gamma_{\mathrm{T_0}^{u_0,g}_{\mathrm{S}}}:=\left\{s\in\Gamma_{\mathrm{T_0}} \mid  u_{0_\tau}(s)=0 \text{ and } \frac{\sigma_{\tau}(u_0)(s)}{g(s)}\in \mathrm{bd} ( {\mathrm{B}(0,1)} )\cap\left(\R\nn(s)\right)^{\perp}\right\}.
\end{array}
$$
\end{myProp}

\begin{proof}
From Hypotheses~\ref{hypo1},~\ref{hypo3} and Lemma~\ref{épidiffgabs}, it follows that
\begin{multline*}
\displaystyle \mathrm{D}_{e}^{2}\Phi(u_{0}\mid E_{0}-u_{0})(\varphi)=\int_{\Gamma_{\mathrm{T_0}^{u_0,g}_{\mathrm{N}}}}p(\theta)\frac{u_{0_\tau}}{\left\|u_{0_\tau}\right\|}\cdot \varphi_\tau+\int_{\Gamma_{\mathrm{T_0}}\backslash\Gamma_{\mathrm{T_0}^{u_0,g}_{\mathrm{N}}}}p(\theta)\frac{\sigma_{\tau}\left(E_0-u_0\right)}{g}\cdot \varphi_\tau\\+\int_{\Gamma_{\mathrm{T_0}}\backslash\Gamma_{\mathrm{T_0}^{u_0,g}_{\mathrm{N}}}}\iota_{\mathrm{N}_{\overline{\mathrm{B}(0,1)}\cap\left(\R\nn(s)\right)^{\perp}}(\frac{\sigma_{\tau}\left(E_0-u_0\right)(s)}{g(s)})}(\varphi(s)) \, \mathrm{d}s,
\end{multline*}
which can be rewritten as
\begin{equation*}
\displaystyle \mathrm{D}_{e}^{2}\Phi(u_{0}\mid E_{0}-u_{0})(\varphi)=\int_{\Gamma_{\mathrm{T_0}^{u_0,g}_{\mathrm{N}}}}p(\theta)\frac{u_{0_\tau}}{\left\|u_{0_\tau}\right\|}\cdot \varphi_\tau+\int_{\Gamma_{\mathrm{T_0}}\backslash\Gamma_{\mathrm{T_0}^{u_0,g}_{\mathrm{N}}}}p(\theta)\frac{\sigma_{\tau}\left(E_0-u_0\right)}{g}\cdot \varphi_\tau+\iota_{\mathcal{K}_0}(\varphi),
\end{equation*}
for all~$\varphi\in\HH^{1}_{\mathrm{D}}(\Omega_0,\R^2)$, where $\mathcal{K}_0$ is the nonempty closed convex subset of $\HH^{1}_{\mathrm{D}}(\Omega_0,\R^2)$ defined by
\begin{multline*}
\mathcal{K}_0:=\biggl\{ \varphi\in \HH^{1}_{\mathrm{D}}(\Omega_0,\R^2)\mid \varphi(s)\in \mathrm{N}_{\overline{\mathrm{B}(0,1)}\cap\left(\R\nn(s)\right)^{\perp}}\left(\frac{\sigma_{\tau}\left(E_0-u_0\right)(s)}{g(s)}\right) \\ \text{ for almost all }s\in\Gamma_{\mathrm{T_0}}\backslash\Gamma_{\mathrm{T_0}^{u_0,g}_{\mathrm{N}}} \biggl\}.
\end{multline*}
Moreover, since $E_0=F_{\Omega_0}\in\HH^{1}_{\mathrm{D}}(\Omega_0,\R^2)$ is solution to the Dirichlet-Neumann problem~\eqref{besoinfortheend} with~$\Omega=\Omega_0$, then~$\sigma_{\tau}(E_0)=0$ \textit{a.e.}\ on $\Gamma_{\mathrm{T_0}}$. Thus it follows that~$\mathcal{K}_0$ is given by~\eqref{eqK0}. Now, since~$\mathrm{D}_{e}^{2}\Phi(u_{0}|E_{0}-u_{0})$ is a proper function on $\HH^{1}_{\mathrm{D}}(\Omega_0,\R^2)$ and the map $t\in\R_{+} \mapsto E_{t} \in \HH^{1}_{\mathrm{D}}(\Omega_0,\R^2)$ is differentiable at $t=0$ with its derivative~$E'_0~\in\HH^{1}_{\mathrm{D}}(\Omega_0,\R^{2})$ being solution to the variational inequality~\eqref{Woderov}, we apply Proposition~\ref{TheoABC2018} to deduce that the map $t\in\mathbb{R}_{+}\mapsto \overline{\overline{u}}_{t}\in \HH^{1}_{\mathrm{D}}(\Omega_0,\R^2)$ is differentiable at $t=0$, and its derivative $\overline{\overline{u}}_{0}'\in\HH^{1}_{\mathrm{D}}(\Omega_0,\R^2)$ satisfies
$$
\displaystyle \overline{\overline{u}}_{0}'=\mathrm{prox}_{\mathrm{D}_{e}^{2}\Phi(u_{0}\mid E_{0}-u_{0})}(E_{0}'),
$$
which, from the definition of the proximal operator (see Definition~\ref{proxi}), leads to
$$
\displaystyle \dual{ E_{0}'-\overline{\overline{u}}'_{0}}{\varphi-\overline{\overline{u}}_{0}'}_{\HH^{1}_{\mathrm{D}}(\Omega_0,\R^2)}\leq \mathrm{D}_{e}^{2}\Phi(u_{0}\mid E_{0}-u_{0})(\varphi) -\mathrm{D}_{e}^{2}\Phi(u_{0}\mid E_{0}-u_{0})(\overline{\overline{u}}_{0}'), 
$$
for all $\varphi\in \HH^{1}_{\mathrm{D}}(\Omega,\R^2)$. Hence we get that
\begin{multline*}
    \dual{E_{0}'-\overline{\overline{u}}'_{0}}{\varphi-\overline{\overline{u}}_{0}'}_{\HH^{1}_{\mathrm{D}}(\Omega_0,\R^2)} \leq +\iota_{\mathcal{K}_0}(\varphi)-\iota_{\mathcal{K}_0}(\overline{\overline{u}}'_0)\\+\int_{\Gamma_{\mathrm{T_0}^{u_0,g}_{\mathrm{N}}}}p(\theta)\frac{u_{0_\tau}}{\left\|u_{0_\tau}\right\|}\cdot \left(\varphi_\tau-\overline{\overline{u}}'_{0_\tau}\right)+\int_{\Gamma_{\mathrm{T_0}}\backslash\Gamma_{\mathrm{T_0}^{u_0,g}_{\mathrm{N}}}}p(\theta)\frac{\sigma_{\tau}\left(E_0-u_0\right)}{g_{0}}\cdot \left(\varphi_\tau-\overline{\overline{u}}'_{0_\tau}\right),
\end{multline*}
for all $\varphi\in\HH^{1}_{\mathrm{D}}(\Omega,\R^2)$. Since~$\varphi_\tau=0$ \textit{a.e.}\ on $\Gamma_{\mathrm{T_0}^{u_0,g}_{\mathrm{D}}}$ for all $\varphi\in\mathcal{K}_0$, one deduces that~$\overline{\overline{u}}'_0\in\mathcal{K}_0$ satisfies
\begin{multline*}
 \dual{\overline{\overline{u}}'_{0}}{\varphi-\overline{\overline{u}}_{0}'}_{\HH^{1}_{\mathrm{D}}(\Omega_0,\R^2)}\geq\int_{\Omega_{0}}\left(\mathrm{div}\left(f\theta^{\top}\right)-\nabla{\theta}f\right)\cdot\left( \varphi-\overline{\overline{u}}'_{0}\right)\\+\int_{\Omega_{0}}\left(\left(\mathrm{A}\mathrm{e}\left({u_0}\right)\right)\nabla{\theta}^{\top}+\mathrm{A}\left(\nabla{u_0}\nabla{\theta}\right)+\mathrm{A}\mathrm{e}\left(\nabla{\theta}^{\top}u_0\right)-\mathrm{div}(\theta)\mathrm{A}\mathrm{e}\left({u_0}\right)\right):\nabla{\left( \varphi-\overline{\overline{u}}'_{0}\right)}\\+\int_{\Omega_0}\mathrm{A}\mathrm{e}\left({u_0}\right):\mathrm{e}\left(\nabla{\theta}^{\top}\left( \varphi-\overline{\overline{u}}'_{0}\right)\right)-\int_{\Gamma_{\mathrm{T_0}^{u_0,g}_{\mathrm{N}}}}p(\theta)\frac{u_{0_\tau}}{\left\|u_{0_\tau}\right\|}\cdot \left( \varphi_\tau-{\overline{\overline{u}}'_{0}}_\tau\right) \\
 +\int_{\Gamma_{\mathrm{T_0}^{u_0,g}_{\mathrm{S}}}}p(\theta)\frac{\sigma_{\tau}\left(u_0\right)}{g}\cdot \left( \varphi_\tau-{\overline{\overline{u}}'_{0}}_\tau\right),
\end{multline*} 
for all $\varphi\in\mathcal{K}_0$. Using the equality~$-\mathrm{div}\left(\mathrm{A}\mathrm{e}(u_0)\right)=f$ in $\HH^1(\Omega_0,\R^2)$ and the divergence formula (see Proposition~\ref{div}), the proof is complete.
\end{proof}

\begin{myRem}\normalfont
Note that Hypothesis~\ref{hypo3} corresponds to the inversion of the symbols~$\mathrm{ME}\text{-}\mathrm{lim}$ and~$\int_{\Gamma_{\mathrm{T}_0}}$ in Equality~\eqref{Delta2}, which is an open question in general. We refer to~\cite[Appendix A]{BCJDC} and~\cite[Remark 2.26]{BCJDC2} for additional comments and some sufficient conditions for this inversion.
\end{myRem}

\begin{myRem}\normalfont
In Proposition~\ref{deriv11} (and in its proof), note that the set~$\mathcal{K}_0$ corresponds to the set~$\mathcal{K}_{u_{0},\frac{\sigma_{\tau}(E_{0}-u_{0})}{g}}$ with the notations introduced in our previous paper~\cite{BCJDC2} (see~\cite[proof of Theorem~2.25]{BCJDC2}).
\end{myRem}

\subsection{Directional material derivative~$\overline{u}'_0$ and directional shape derivative~$u'_0$}\label{section23a}

From Proposition~\ref{deriv11} and since $\overline{u}_{t}=\left(\mathrm{I}+t\nabla{\theta}^{\top}\right)^{-1}\overline{\overline{u}}_t$ for all~$t \geq 0$, it is possible now to state and prove the first main result of this paper that characterizes the directional material derivative~$\overline{u}'_0$.

\begin{myTheorem}[Directional material derivative~$\overline{u}'_0$]\label{materialderiv1}
Consider the framework of Proposition~\ref{deriv11}. Then the map
$t\in\R_{+} \mapsto \overline{u}_{t} \in \HH^{1}_{\mathrm{D}}(\Omega_0,\R^2)$ is differentiable at $t=0$ and its derivative $\overline{u}'_{0} \in\mathcal{K}_0-\nabla{\theta}^{\top}u_0$ (that is, the directional material derivative) is the unique solution to the variational inequality
\begin{multline}\label{inequalityofmaterialderiv1}
\dual{\overline{u}'_{0}}{v-\overline{u}'_{0}}_{\HH^{1}_{\mathrm{D}}(\Omega_0,\R^{2})}\geq-\int_{\Omega_0}\mathrm{div}\left(\mathrm{div}\left(\mathrm{A}\mathrm{e}(u_0)\right)\theta^{\top}\right)\cdot\left(v-\overline{u}'_0\right)\\+\int_{\Omega_0}\left(\left(\mathrm{A}\mathrm{e}\left({u_0}\right)\right)\nabla{\theta}^{\top}+\mathrm{A}\left(\nabla{u_0}\nabla{\theta}\right)-\mathrm{div}(\theta)\mathrm{A}\mathrm{e}\left({u_0}\right)\right):\nabla{\left(v-\overline{u}'_{0}\right)}\\+\dual{\mathrm{A}\mathrm{e}(u_0)\nn}{\nabla{\theta}^{\top}\left(v-\overline{u}'_{0}\right)}_{\HH^{-1/2}(\Gamma_0,\R^2)\times\HH^{1/2}(\Gamma_0,\R^2)}-\int_{\Gamma_{\mathrm{T_0}^{u_0,g}_{\mathrm{N}}}}p(\theta)\frac{u_{0_\tau}}{\left\|u_{0_\tau}\right\|}\cdot \left(v_\tau-\overline{u}'_{0_\tau}\right)\\+\int_{\Gamma_{\mathrm{T_0}^{u_0,g}_{\mathrm{S}}}}p(\theta)\frac{\sigma_{\tau}(u_0)}{g}\cdot \left(v_\tau-\overline{u}'_{0_\tau}\right), \qquad \forall v\in\mathcal{K}_0-\nabla{\theta}^{\top}u_0, 
\end{multline}
where
\begin{multline*}
    \mathcal{K}_0-\nabla{\theta}^{\top}u_0= \biggl\{ v\in \HH^{1}_{\mathrm{D}}(\Omega_0,\R^2)\mid v_\tau=-\left(\nabla{\theta}^{\top}u_0\right)_\tau \text{ \textit{a.e.}\ on } \Gamma_{\mathrm{T_0}^{u_0,g}_{\mathrm{D}}} \\ \text{ and } \left(v_\tau+\left(\nabla{\theta}^{\top}u_0\right)_\tau\right)\in\R_{-}\frac{\sigma_{\tau}(u_0)}{g} \text{ \textit{a.e.}\ on } \Gamma_{\mathrm{T_0}^{u_0,g}_{\mathrm{S}}} \biggl\}.    
\end{multline*}
\end{myTheorem}

\begin{proof}
Since $\overline{u}_{t}=\left(\mathrm{I}+t\nabla{\theta}^{\top}\right)^{-1}\overline{\overline{u}}_t$ for all~$t \ge 0$, one deduces from Proposition~\ref{deriv11} that the map~$t\in\R_{+} \mapsto \overline{u}_{t} \in \HH^{1}_{\mathrm{D}}(\Omega_0,\R^2)$ is differentiable at $t=0$ with $\overline{u}'_{0}=\overline{\overline{u}}'_{0}-\nabla{\theta}^{\top}u_0\in\HH^{1}_{\mathrm{D}}(\Omega_0,\R^2)$.
Moreover, from the variational inequality satisfied by $\overline{\overline{u}}'_0$, one deduces that
\begin{multline*}
\dual{\overline{u}'_0+\nabla{\theta}^{\top}u_0}{\varphi-\nabla{\theta}^{\top}u_0-\overline{u}'_0 }_{\HH^{1}_{\mathrm{D}}(\Omega_0,\R^{2})}\geq-\int_{\Omega_0}\mathrm{div}\left(\mathrm{div}\left(\mathrm{A}\mathrm{e}(u_0)\right)\theta^{\top}\right)\cdot\left(\varphi-\nabla{\theta}^{\top}u_0-\overline{u}'_0\right)\\+\int_{\Omega_0}\left(\left(\mathrm{A}\mathrm{e}\left({u_0}\right)\right)\nabla{\theta}^{\top}+\mathrm{A}\left(\nabla{u_0}\nabla{\theta}\right)+\mathrm{A}\mathrm{e}{\left(\nabla{\theta}^{\top}u_0\right)}-\mathrm{div}(\theta)\mathrm{A}\mathrm{e}\left({u_0}\right)\right):\nabla{\left(\varphi-\nabla{\theta}^{\top}u_0-\overline{u}'_0\right)}\\+\dual{\mathrm{A}\mathrm{e}(u_0)\nn}{\nabla{\theta}^{\top}\left(\varphi-\nabla{\theta}^{\top}u_0-\overline{u}'_0\right)}_{\HH^{-1/2}(\Gamma_0,\R^2)\times\HH^{1/2}(\Gamma_0,\R^2)}\\-\int_{\Gamma_{\mathrm{T_0}^{u_0,g}_{\mathrm{N}}}}p(\theta)\frac{u_{0_\tau}}{\left\|u_{0_\tau}\right\|}\cdot \left(\varphi_\tau-\left(\nabla{\theta}^{\top}u_0\right)_\tau-{\overline{u}'_0}_\tau\right)+\int_{\Gamma_{\mathrm{T_0}^{u_0,g}_{\mathrm{S}}}}p(\theta)\frac{\sigma_{\tau}(u_0)}{g}\cdot \left(\varphi_\tau-\left(\nabla{\theta}^{\top}u_0\right)_\tau-{\overline{u}'_0}_\tau\right),
\end{multline*}
for all $\varphi\in\mathcal{K}_0$, and this is also
\begin{multline*}
\dual{\overline{u}'_0+\nabla{\theta}^{\top}u_0}{v-\overline{u}'_0 }_{\HH^{1}_{\mathrm{D}}(\Omega_0,\R^{2})}\geq-\int_{\Omega_0}\mathrm{div}\left(\mathrm{div}\left(\mathrm{A}\mathrm{e}(u_0)\right)\theta^{\top}\right)\cdot\left(v-\overline{u}'_0\right)\\+\int_{\Omega_0}\left(\left(\mathrm{A}\mathrm{e}\left({u_0}\right)\right)\nabla{\theta}^{\top}+\mathrm{A}\left(\nabla{u_0}\nabla{\theta}\right)+\mathrm{A}\mathrm{e}{\left(\nabla{\theta}^{\top}u_0\right)}-\mathrm{div}(\theta)\mathrm{A}\mathrm{e}\left({u_0}\right)\right):\nabla{\left(v-\overline{u}'_0\right)}\\+\dual{\mathrm{A}\mathrm{e}(u_0)\nn}{\nabla{\theta}^{\top}\left(v-\overline{u}'_0\right)}_{\HH^{-1/2}(\Gamma_0,\R^2)\times\HH^{1/2}(\Gamma_0,\R^2)}\\-\int_{\Gamma_{\mathrm{T_0}^{u_0,g}_{\mathrm{N}}}}p(\theta)\frac{u_{0_\tau}}{\left\|u_{0_\tau}\right\|}\cdot \left(v_\tau-{\overline{u}'_0}_\tau\right)+\int_{\Gamma_{\mathrm{T_0}^{u_0,g}_{\mathrm{S}}}}p(\theta)\frac{\sigma_{\tau}(u_0)}{g}\cdot \left(v_\tau-{\overline{u}'_0}_\tau\right),
\end{multline*}
for all $v\in\mathcal{K}_0-\nabla{\theta}^{\top}u_0$, which concludes the proof.
\end{proof}

The presentation of Theorem~\ref{materialderiv1} can be improved under additional regularity assumptions.

\begin{myCor}\label{materialderiv2}
Consider the framework of Proposition~\ref{deriv11} with the additional assumption that~$u_{0}\in\HH^{3}(\Omega_{0},\R^2)$. Then the directional material derivative
$\overline{u}'_0\in\mathcal{K}_0-\nabla{\theta}^{\top}u_0$ is the unique weak solution to the tangential Signorini problem given by
\begin{equation*}
{
{\arraycolsep=2pt
\left\{
\begin{array}{rcll}
-\mathrm{div}\left(\mathrm{A}\mathrm{e}(\overline{u}'_0)\right)+\mathrm{div}\left(\mathrm{A}\mathrm{e}\left(\nabla{u_0}\theta\right)\right) &=& 0  & \text{ in } \Omega_0 , \\[5pt]
\overline{u}'_0 & = & 0  & \text{ on } \Gamma_{\mathrm{D}} ,\\[5pt]
\sigma_{\nn}(\overline{u}'_0) - {\xi^{m}(\theta)}_\nn & = & 0  & \text{ on } \Gamma_{\mathrm{T_0}} ,\\[5pt]
\sigma_\tau (\overline{u}'_0) +p(\theta)\frac{u_{0_\tau}}{\left\|u_{0_\tau}\right\|} -{\xi^{m}(\theta)}_{\tau} & = & 0  & \text{ on } \Gamma_{\mathrm{T_0}^{u_0,g}_{\mathrm{N}}} ,\\[5pt]
{\overline{u}'_0}_{\tau} +\left(\nabla{\theta}^{\top}u_0\right)_\tau & = & 0  & \text{ on } \Gamma_{\mathrm{T_0}^{u_0,g}_{\mathrm{D}}},\\[5pt]
\left({\overline{u}'_0}_{\tau}+\left(\nabla{\theta}^{\top}u_0\right)_\tau\right)\in\R_{-}\frac{\sigma_{\tau}(u_0)}{g}  \\[5pt] \text{and } \left(\sigma_\tau (\overline{u}'_0)-p (\theta)\frac{\sigma_{\tau}(u_0)}{g}-{\xi^{m}(\theta)}_{\tau}\right)\cdot \frac{\sigma_{\tau}(u_0)}{g}\leq0 \\[5pt] \text{and }\left({\overline{u}'_0}_{\tau}+\left(\nabla{\theta}^{\top}u_0\right)_\tau\right)\cdot\left(\sigma_\tau (\overline{u}'_0)-p(\theta) \frac{\sigma_{\tau}(u_0)}{g}-{\xi^{m}(\theta)}_{\tau}\right)  & = & 0  & \text{ on } \Gamma_{\mathrm{T_0}^{u_0,g}_{\mathrm{S}}}.
\end{array}
\right.}}
\end{equation*}
where $\xi^m(\theta):=\left((\mathrm{A}\mathrm{e}\left({u_0}\right))\nabla{\theta}^{\top}+\mathrm{A}(\nabla{u_0}\nabla{\theta})+(\nabla{\theta}-\mathrm{div}(\theta)\mathrm{I})\mathrm{A}\mathrm{e}\left({u_0}\right)\right)\nn\in\LL^2(\Gamma_{\mathrm{T}_0},\R^2)$.
\end{myCor}

\begin{proof}
Since $u_{0}\in\HH^{2}(\Omega_{0},\R^2)$ and $\theta \in \mathcal{C}_{\mathrm{D}}^{2,\infty}(\R^{2},\R^{2})$, it holds that 
    $$
\mathrm{div}\left(\left(\mathrm{A}\mathrm{e}\left({u_0}\right)\right)\nabla{\theta}^{\top}+\mathrm{A}\left(\nabla{u_0}\nabla{\theta}\right)-\mathrm{div}(\theta)\mathrm{A}\mathrm{e}\left({u_0}\right)\right)\in\LL^2(\Omega_0,\R^2).
$$
Thus, using the divergence formula (see Proposition~\ref{div}) in Inequality~\eqref{inequalityofmaterialderiv1}, we get that
\begin{multline}\label{formulderivmatH2}
\dual{\overline{u}'_0}{v-\overline{u}'_0 }_{\HH^{1}_{\mathrm{D}}(\Omega_0,\R^2)}\geq \int_{\Gamma_{\mathrm{T}_0}}\xi^{m}(\theta)\cdot\left(v-\overline{u}'_0\right)\\ -\int_{\Omega_0}\mathrm{div}\left(\mathrm{div}(\mathrm{A}\mathrm{e}(u_0))\theta^{\top}+\left(\mathrm{A}\mathrm{e}\left({u_0}\right)\right)\nabla{\theta}^{\top}+\mathrm{A}\left(\nabla{u_0}\nabla{\theta}\right)-\mathrm{div}(\theta)\mathrm{A}\mathrm{e}\left({u_0}\right)\right)\cdot\left(v-\overline{u}'_0\right)\\-\int_{\Gamma_{\mathrm{T_0}^{u_0,g}_{\mathrm{N}}}}p(\theta)\frac{u_{0_\tau}}{\left\|u_{0_\tau}\right\|}\cdot \left(v_\tau-\overline{u}'_{0_\tau}\right)+\int_{\Gamma_{\mathrm{T_0}^{u_0,g}_{\mathrm{S}}}}p(\theta)\frac{\sigma_{\tau}(u_0)}{g}\cdot \left(v_\tau-\overline{u}'_{0_\tau}\right),
\end{multline}
for all $v\in\mathcal{K}_0-\nabla{\theta}^{\top}u_0$. Furthermore, one has $\mathrm{div}\left(\mathrm{A}\mathrm{e}\left(\nabla{u_0}\theta\right)\right)\in\LL^2(\Omega_0,\R^2)$ from the fact that~$u_0\in\HH^3(\Omega_0,\R^2)$. Thus, using the equality
$$
\mathrm{div}\left(\mathrm{A}\mathrm{e}\left(\nabla{u_0}\theta\right)\right)=\mathrm{div}\left(\mathrm{div}(\mathrm{A}\mathrm{e}(u_0))\theta^{\top}+\left(\mathrm{A}\mathrm{e}\left({u_0}\right)\right)\nabla{\theta}^{\top}+\mathrm{A}\left(\nabla{u_0}\nabla{\theta}\right)-\mathrm{div}(\theta)\mathrm{A}\mathrm{e}\left({u_0}\right)\right),
$$
in $\LL^2(\Omega_0,\R^2)$, it follows that
\begin{multline*}
\dual{\overline{u}'_0}{v-\overline{u}'_0 }_{\HH^{1}_{\mathrm{D}}(\Omega_0,\R^2)}\geq -\int_{\Omega_0}\mathrm{div}\left(\mathrm{A}\mathrm{e}\left(\nabla{u_0}\theta\right)\right)\cdot\left(v-\overline{u}'_0\right)
+
\int_{\Gamma_{\mathrm{T}_0}}\xi^{m}(\theta)\cdot\left(v-\overline{u}'_0\right)\\-\int_{\Gamma_{\mathrm{T_0}^{u_0,g}_{\mathrm{N}}}}p(\theta)\frac{u_{0_\tau}}{\left\|u_{0_\tau}\right\|}\cdot \left(v_\tau-\overline{u}'_{0_\tau}\right)+\int_{\Gamma_{\mathrm{T_0}^{u_0,g}_{\mathrm{S}}}}p(\theta)\frac{\sigma_{\tau}(u_0)}{g}\cdot \left(v_\tau-\overline{u}'_{0_\tau}\right),
\end{multline*}
for all $v\in\mathcal{K}_0-\nabla{\theta}^{\top}u_0$, which corresponds to the weak formulation of the expected tangential Signorini problem (see \cite[Section 2.1.2]{BCJDC2} for details).
\end{proof}

\begin{myRem}\normalfont
Note that, from the proof of Corollary~\ref{materialderiv2}, one can get, under the weaker assumption~$u_0 \in \HH^2(\Omega_0,\R^2)$, that the directional material derivative $\overline{u}'_0$ is the solution to the variational inequality~\eqref{formulderivmatH2} which is, from \cite[Section 2.1.2]{BCJDC2}, the weak formulation of a tangential Signorini problem, with the source term given by~$-\mathrm{div}(\mathrm{div}(\mathrm{A}\mathrm{e}(u_0))\theta^{\top}+(\mathrm{A}\mathrm{e}(u_0))\nabla{\theta}^{\top}+\mathrm{A}(\nabla{u_0}\nabla{\theta})-\mathrm{div}(\theta)\mathrm{A}\mathrm{e}(u_0))\in\LL^2(\Omega_0,\R^2)$.
\end{myRem}

Thanks to Corollary~\ref{materialderiv2}, we are now in a position to characterize the directional shape derivative~$u'_0$.

\begin{myCor}[Directional shape derivative~$u'_0$]\label{shapederiv1}
Consider the framework of Proposition~\ref{deriv11} with the additional assumptions that~$u_{0}\in\HH^{3}(\Omega_{0},\R^2)$ and that $\Gamma_0$ is of class $\mathcal{C}^3$. Then the directional shape derivative, defined by~$u'_{0}:=\overline{u}'_0-\nabla{u_0}\theta\in\mathcal{K}_0-\nabla{\theta}^{\top}u_0-\nabla{u_0}\theta$, is the unique weak solution to the tangential Signorini problem
\begin{equation*}
{
{\arraycolsep=2pt
\left\{
\begin{array}{rcll}
-\mathrm{div}\left(\mathrm{A}\mathrm{e}(u'_0)\right) &=& 0  & \text{ in } \Omega_0 , \\
u'_0 & = & 0  & \text{ on } \Gamma_{\mathrm{D}} ,\\[5pt]
\sigma_{\nn}(u'_0) -{\xi^{s}(\theta)}_\nn & = & 0  & \text{ on } \Gamma_{\mathrm{T_0}} ,\\[5pt]
\sigma_\tau (u'_0) +p(\theta)\frac{u_{0_\tau}}{\left\|u_{0_\tau}\right\|} -{\xi^{s}(\theta)}_{\tau} & = & 0  & \text{ on } \Gamma_{\mathrm{T_0}^{u_0,g}_{\mathrm{N}}} ,\\[5pt]
{u'_0}_{\tau} - W(\theta)_\tau & = & 0  & \text{ on } \Gamma_{\mathrm{T_0}^{u_0,g}_{\mathrm{D}}},\\[5pt]
\left({u'_0}_{\tau}-W(\theta)_\tau\right)\in\R_{-}\frac{\sigma_{\tau}(u_0)}{g} \\ [5pt]
\text{and } \left(\sigma_\tau (\overline{u}'_0)-p(\theta) \frac{\sigma_{\tau}(u_0)}{g}-{\xi^{s}(\theta)}_{\tau}\right)\cdot \frac{\sigma_{\tau}(u_0)}{g}\leq0 \\ [5pt]
\text{and } \left({u'_0}_{\tau}-W(\theta)_\tau\right)\cdot\left(\sigma_\tau (\overline{u}'_0)-p(\theta) \frac{\sigma_{\tau}(u_0)}{g}-{\xi^{s}(\theta)}_{\tau}\right)  & = & 0  & \text{ on } \Gamma_{\mathrm{T_0}^{u_0,g}_{\mathrm{S}}}.
\end{array}
\right.}}
\end{equation*}
where $W(\theta):=-\nabla{\theta}^{\top}u_0-\nabla{u_0}\theta\in\HH^{1}(\Omega_0,\R^2)$,
\begin{multline*}
    \xi^s(\theta):=\theta \cdot \nn\left(\partial_{\nn}\left(\mathrm{A}\mathrm{e}\left({u_0}\right)\nn\right)-\partial_{\nn}\left(\mathrm{A}\mathrm{e}\left({u_0}\right)\right)\nn\right)+\mathrm{A}\mathrm{e}\left({u_0}\right)\nabla_{\tau}\left(\theta\cdot\nn\right)-\nabla{\left(\mathrm{A}\mathrm{e}\left({u_0}\right)\nn\right)}\theta\\+\left(\nabla{\theta}-\mathrm{div}_{\tau}(\theta)\mathrm{I}\right)\mathrm{A}\mathrm{e}\left({u_0}\right)\nn\in\LL^2(\Gamma_{\mathrm{T}_0},\R^2),
    \end{multline*}
where $\partial_{\nn}\left(\mathrm{A}\mathrm{e}\left({u_0}\right)\nn\right):=\nabla{\left(\mathrm{A}\mathrm{e}\left({u_0}\right)\nn\right)}\nn$ stands for the normal derivative of $\mathrm{A}\mathrm{e}\left({u_0}\right)\nn$, and~$\partial_{\nn}\left(\mathrm{A}\mathrm{e}\left({u_0}\right)\right)$ is the matrix whose the $i$-th line is the transpose of the vector $\partial_{\nn}\left(\mathrm{A}\mathrm{e}\left({u_0}\right)_{i}\right):=\nabla{\left(\mathrm{A}\mathrm{e}\left({u_0}\right)_{i}\right)}\nn$, where~$\mathrm{A}\mathrm{e}\left({u_0}\right)_{i}$ is the transpose of the~$i$-th line of the matrix $\mathrm{A}\mathrm{e}\left({u_0}\right)$, for all~$i\in \{ 1,2 \}$.

\end{myCor}
\begin{proof}
Since $u'_{0}:=\overline{u}'_0-\nabla{u_0}\theta$, one deduces from the weak formulation of $\overline{u}'_0$ and the divergence formula (see Proposition~\ref{div}) that
\begin{multline*}
\dual{u'_0}{v-u'_0}_{\HH^{1}_{\mathrm{D}}(\Omega_0,\R^2)}\geq \\ \int_{\Omega_0}\left(\mathrm{div}\left(\mathrm{A}\mathrm{e}(u_0)\right)\theta^{\top}+\left(\mathrm{A}\mathrm{e}(u_0)\right)\nabla{\theta}^{\top}+\mathrm{A}\left(\nabla{u_0}\nabla{\theta}\right)-\mathrm{A}\mathrm{e}\left({\nabla{u_0}\theta}\right)\right):\nabla{\left(v-u'_{0}\right)}\\-\int_{\Omega_0}\mathrm{div}(\theta)\mathrm{A}\mathrm{e}(u_0):\mathrm{e}\left(v-u'_{0}\right)+\int_{\Gamma_{\mathrm{T}_0}}\mathrm{A}\mathrm{e}(u_0)\nn\cdot\nabla{\theta}^{\top}\left(v-u'_{0}\right)-\int_{\Gamma_{0}}\left(\theta\cdot\nn\right) \mathrm{div}\left(\mathrm{A}\mathrm{e}(u_0)\right)\cdot\left(v-u'_{0}\right)\\-\int_{\Gamma_{\mathrm{T_0}^{u_0,g}_{\mathrm{N}}}}p(\theta)\frac{u_{0_\tau}}{\left\|u_{0_\tau}\right\|}\cdot \left(v_\tau-u'_{0_\tau}\right)+\int_{\Gamma_{\mathrm{T_0}^{u_0,g}_{\mathrm{S}}}}p(\theta)\frac{\sigma_{\tau}(u_0)}{g}\cdot \left(v_\tau-u'_{0_\tau}\right),
\end{multline*}
for all $v\in\mathcal{K}_0-\nabla{\theta}^{\top}u_0-\nabla{u_0}\theta$. Moreover, one has
$$\int_{\Omega_0}\mathrm{div}\left(\mathrm{A}\mathrm{e}(u_0)\right)\theta^{\top}:\nabla{\varphi}=\int_{\Omega_0}\mathrm{div}\left(\mathrm{A}\mathrm{e}(u_0)\right)\cdot\nabla{\varphi}\theta=-\int_{\Omega_0}\mathrm{A}\mathrm{e}(u_0):\nabla{\left(\nabla{\varphi}\theta\right)}+\int_{\Gamma_0}\mathrm{A}\mathrm{e}(u_0)\nn\cdot\nabla{\varphi}\theta,
$$
and also
$$
-\int_{\Omega_0}\mathrm{div}(\theta)\mathrm{A}\mathrm{e}(u_0):\mathrm{e}(\varphi)=\int_{\Omega_0}\theta\cdot\nabla{\left(\mathrm{A}\mathrm{e}(u_0):\mathrm{e}(\varphi)\right)}-\int_{\Gamma_0}\theta\cdot\nn\left(\mathrm{A}\mathrm{e}(u_0):\mathrm{e}(\varphi)\right),
$$
for all $\varphi\in\mathcal{C}^{\infty}(\overline{\Omega_{0}},\R^2)$.
Therefore, using the equality
$$
\left(\left(\mathrm{A}\mathrm{e}(u_0)\right)\nabla{\theta}^{\top}+\mathrm{A}\left(\nabla{u_0}\nabla{\theta}\right)-\mathrm{A}\mathrm{e}\left(\nabla{u_0}\theta\right)\right):\nabla{\varphi}+\theta\cdot\nabla{\left(\mathrm{A}\mathrm{e}(u_0):\mathrm{e}(\varphi)\right)}-\mathrm{A}\mathrm{e}(u_0):\nabla{\left(\nabla{\varphi}\theta\right)}=0,
$$ which holds true \textit{a.e.}\ on $\Omega_0$, one deduces from the divergence formula that
\begin{multline*}
\int_{\Omega_0}\left(\mathrm{div}\left(\mathrm{A}\mathrm{e}(u_0)\right)\theta^{\top}+\left(\mathrm{A}\mathrm{e}(u_0)\right)\nabla{\theta}^{\top}+\mathrm{A}\left(\nabla{u_0}\nabla{\theta}\right)-\mathrm{A}\mathrm{e}\left(\nabla{u_0}\theta\right)\right):\nabla{\varphi}\\-\int_{\Omega_0}\mathrm{div}(\theta)\mathrm{A}\mathrm{e}(u_0):\mathrm{e}\left(\varphi\right)+\int_{\Gamma_0}\nabla{\theta}\left(\mathrm{A}\mathrm{e}(u_0)\nn\right)\cdot \varphi-\int_{\Gamma_{0}}\left(\theta\cdot\nn\right) \mathrm{div}\left(\mathrm{A}\mathrm{e}(u_0)\right)\cdot \varphi-\int_{\Gamma_{\mathrm{T_0}^{u_0,g}_{\mathrm{N}}}}p(\theta)\frac{u_{0_\tau}}{\left\|u_{0_\tau}\right\|}\cdot \varphi_\tau\\+\int_{\Gamma_{\mathrm{T_0}^{u_0,g}_{\mathrm{S}}}}p(\theta)\frac{\sigma_{\tau}(u_0)}{g}\cdot \varphi_\tau=\int_{\Gamma_0}\theta\cdot\nn\left(-\mathrm{A}\mathrm{e}(u_0):\mathrm{e}(\varphi)-\mathrm{div}\left(\mathrm{A}\mathrm{e}(u_0)\right)\cdot \varphi\right)\\+\int_{\Gamma_0}\nabla{\varphi}^{\top}(\mathrm{A}\mathrm{e}(u_0)\nn)\cdot\theta+\nabla{\theta}(\mathrm{A}\mathrm{e}(u_0)\nn)\cdot \varphi-\int_{\Gamma_{\mathrm{T_0}^{u_0,g}_{\mathrm{N}}}}p(\theta)\frac{u_{0_\tau}}{\left\|u_{0_\tau}\right\|}\cdot \varphi_\tau+\int_{\Gamma_{\mathrm{T_0}^{u_0,g}_{\mathrm{S}}}}p(\theta)\frac{\sigma_{\tau}(u_0)}{g}\cdot \varphi_\tau,
\end{multline*}
for all $\varphi\in\mathcal{C}^{\infty}(\overline{\Omega_{0}},\R^2)$. Furthermore, since $\Gamma_0$ is of class $\mathcal{C}^{3}$ and $u_0\in\HH^3(\Omega_0,\R^2)$, $\mathrm{A}\mathrm{e}(u_0)\nn$ can be extended into a function defined in $\Omega_{0}$ such that~$\mathrm{A}\mathrm{e}(u_0)\nn\in\HH^{2}(\Omega_{0},\R^2)$. Thus, it holds that~$\mathrm{A}\mathrm{e}(u_0)\nn\cdot \varphi\in\mathrm{W}^{2,1}(\Omega_{0},\R^2)$, for all $\varphi\in\mathcal{C}^{\infty}(\overline{\Omega_{0}},\R^2)$, and one can use Proposition~\ref{intbord} to get that
\begin{multline*}
\int_{\Gamma_0}\theta\cdot\nn\left(-\mathrm{A}\mathrm{e}(u_0):\mathrm{e}(\varphi)-\mathrm{div}\left(\mathrm{A}\mathrm{e}(u_0)\right)\cdot \varphi\right)+\int_{\Gamma_0}\nabla{\varphi}^{\top}(\mathrm{A}\mathrm{e}(u_0)\nn)\cdot\theta+\nabla{\theta}(\mathrm{A}\mathrm{e}(u_0)\nn)\cdot \varphi\\-\int_{\Gamma_{\mathrm{T_0}^{u_0,g}_{\mathrm{N}}}}p(\theta)\frac{u_{0_\tau}}{\left\|u_{0_\tau}\right\|}\cdot \varphi_\tau+\int_{\Gamma_{\mathrm{T_0}^{u_0,g}_{\mathrm{S}}}}p(\theta)\frac{\sigma_{\tau}(u_0)}{g}\cdot \varphi_\tau\\=\int_{\Gamma_0}\theta\cdot\nn\left(-\mathrm{A}\mathrm{e}(u_0):\mathrm{e}(\varphi)-\mathrm{div}\left(\mathrm{A}\mathrm{e}(u_0)\right)\cdot \varphi+\partial_{\nn}\left(\mathrm{A}\mathrm{e}(u_0)\nn\cdot \varphi\right)+H\mathrm{A}\mathrm{e}(u_0)\nn\cdot \varphi\right)
\\-\int_{\Gamma_0}\left(\nabla{\left(\mathrm{A}\mathrm{e}(u_0)\nn\right)}\theta-\nabla{\theta}(\mathrm{A}\mathrm{e}(u_0)\nn)+\mathrm{div}_{\tau}(\theta)\mathrm{A}\mathrm{e}(u_0)\nn\right)\cdot \varphi-\int_{\Gamma_{\mathrm{T_0}^{u_0,g}_{\mathrm{N}}}}p(\theta)\frac{u_{0_\tau}}{\left\|u_{0_\tau}\right\|}\cdot \varphi_\tau\\+\int_{\Gamma_{\mathrm{T_0}^{u_0,g}_{\mathrm{S}}}}p(\theta)\frac{\sigma_{\tau}(u_0)}{g}\cdot \varphi_\tau,
\end{multline*}
where $H$ is the mean curvature of $\Gamma_{0}$. By Proposition~\ref{beltrami} it follows that
\begin{equation*}
\int_{\Gamma_0}\theta\cdot\nn\left(-\mathrm{div}\left(\mathrm{A}\mathrm{e}(u_0)\right) +H\mathrm{A}\mathrm{e}(u_0)\nn\right)\cdot \varphi = \int_{\Gamma_0}\mathrm{A}\mathrm{e}(u_0):\nabla_{\tau}\left(\varphi\left(\theta\cdot\nn\right)\right)-\left(\theta\cdot\nn\right)\partial_{\nn}\left(\mathrm{A}\mathrm{e}(u_0)\right)\nn\cdot \varphi,
\end{equation*}
for all~$\varphi\in\mathcal{C}^{\infty}(\overline{\Omega_{0}},\R^2)$. Therefore, using the following two equalities
$$\mathrm{A}\mathrm{e}(u_0):\nabla_{\tau}\left(\varphi\left(\theta\cdot\nn\right)\right)=\theta\cdot\nn\left(\mathrm{A}\mathrm{e}(u_0):\nabla_{\tau}\varphi\right)+\mathrm{A}\mathrm{e}(u_0)\nabla_{\tau}\left(\theta\cdot\nn\right)\cdot \varphi, \text{ \textit{a.e.}\ on } \Gamma_0,
$$
and
$$\mathrm{A}\mathrm{e}(u_0):\nabla_{\tau}\varphi=\mathrm{A}\mathrm{e}(u_0):e(\varphi)-\nabla{\varphi}^{\top}(\mathrm{A}\mathrm{e}(u_0)\nn)\cdot\nn \text{ \textit{a.e} on } \Gamma_0,
$$ 
one gets
\begin{multline*}
    \int_{\Gamma_0}\theta\cdot\nn\left(-\mathrm{A}\mathrm{e}(u_0):\mathrm{e}(\varphi)-\mathrm{div}\left(\mathrm{A}\mathrm{e}(u_0)\right)\cdot \varphi+\partial_{\nn}\left(\mathrm{A}\mathrm{e}(u_0)\nn\cdot \varphi\right)+H\mathrm{A}\mathrm{e}(u_0)\nn\cdot \varphi\right)
\\-\int_{\Gamma_0}\left(\nabla{\left(\mathrm{A}\mathrm{e}(u_0)\nn\right)}\theta-\nabla{\theta}(\mathrm{A}\mathrm{e}(u_0)\nn)+\mathrm{div}_{\tau}(\theta)\mathrm{A}\mathrm{e}(u_0)\nn\right)\cdot \varphi -\int_{\Gamma_{\mathrm{T_0}^{u_0,g}_{\mathrm{N}}}}p(\theta)\frac{u_{0_\tau}}{\left\|u_{0_\tau}\right\|}\cdot \varphi_\tau\\+\int_{\Gamma_{\mathrm{T_0}^{u_0,g}_{\mathrm{S}}}}p(\theta)\frac{\sigma_{\tau}(u_0)}{g}\cdot \varphi_\tau =\int_{\Gamma_0}\left(\theta \cdot \nn\left(\partial_{\nn}\left(\mathrm{A}\mathrm{e}(u_0)\nn\right)-\partial_{\nn}\left(\mathrm{A}\mathrm{e}(u_0)\right)\nn\right)+\mathrm{A}\mathrm{e}(u_0)\nabla_{\tau}\left(\theta\cdot\nn\right)\right)\cdot \varphi\\
+\int_{\Gamma_0}\left(-\nabla{\left(\mathrm{A}\mathrm{e}(u_0)\nn\right)}\theta+\left(\nabla{\theta}-\mathrm{div}_{\tau}(\theta)\mathrm{I}\right)\mathrm{A}\mathrm{e}(u_0)\nn\right)\cdot \varphi-\int_{\Gamma_{\mathrm{T_0}^{u_0,g}_{\mathrm{N}}}}p(\theta)\frac{u_{0_\tau}}{\left\|u_{0_\tau}\right\|}\cdot \varphi_\tau+\int_{\Gamma_{\mathrm{T_0}^{u_0,g}_{\mathrm{S}}}}p(\theta)\frac{\sigma_{\tau}(u_0)}{g}\cdot \varphi_\tau,
\end{multline*}
and thus
\begin{multline*}
    \int_{\Omega_0}\left(\mathrm{div}\left(\mathrm{A}\mathrm{e}(u_0)\right)\theta^{\top}+\left(\mathrm{A}\mathrm{e}(u_0)\right)\nabla{\theta}^{\top}+\mathrm{A}\left(\nabla{u_0}\nabla{\theta}\right)-\mathrm{A}\mathrm{e}\left({\nabla{u_0}\theta}\right)\right):\nabla{\varphi}-\int_{\Omega_0}\mathrm{div}(\theta)\mathrm{A}\mathrm{e}(u_0):\mathrm{e}\left(\varphi\right)\\+\int_{\Gamma_{\mathrm{T}_0}}\mathrm{A}\mathrm{e}(u_0)\nn\cdot\nabla{\theta}^{\top}\varphi-\int_{\Gamma_{0}}\left(\theta\cdot\nn\right) \mathrm{div}\left(\mathrm{A}\mathrm{e}(u_0)\right)\cdot\varphi-\int_{\Gamma_{\mathrm{T_0}^{u_0,g}_{\mathrm{N}}}}p(\theta)\frac{u_{0_\tau}}{\left\|u_{0_\tau}\right\|}\cdot \varphi_\tau+\int_{\Gamma_{\mathrm{T_0}^{u_0,g}_{\mathrm{S}}}}p(\theta)\frac{\sigma_{\tau}(u_0)}{g}\cdot \varphi_\tau \\=\int_{\Gamma_0}\left(\theta \cdot \nn\left(\partial_{\nn}\left(\mathrm{A}\mathrm{e}(u_0)\nn\right)-\partial_{\nn}\left(\mathrm{A}\mathrm{e}(u_0)\right)\nn\right)+\mathrm{A}\mathrm{e}(u_0)\nabla_{\tau}\left(\theta\cdot\nn\right)\right)\cdot \varphi\\
+\int_{\Gamma_0}\left(-\nabla{\left(\mathrm{A}\mathrm{e}(u_0)\nn\right)}\theta+\left(\nabla{\theta}-\mathrm{div}_{\tau}(\theta)\mathrm{I}\right)\mathrm{A}\mathrm{e}(u_0)\nn\right)\cdot \varphi-\int_{\Gamma_{\mathrm{T_0}^{u_0,g}_{\mathrm{N}}}}p(\theta)\frac{u_{0_\tau}}{\left\|u_{0_\tau}\right\|}\cdot \varphi_\tau+\int_{\Gamma_{\mathrm{T_0}^{u_0,g}_{\mathrm{S}}}}p(\theta)\frac{\sigma_{\tau}(u_0)}{g}\cdot \varphi_\tau,
\end{multline*}
for all~$\varphi\in\mathcal{C}^{\infty}(\overline{\Omega_{0}},\R^2)$. Finally, one deduces from the density of $\mathcal{C}^{\infty}(\overline{\Omega_{0}},\R^2)$ in $\HH^1(\Omega_0,\R^{2})$ that 
\begin{multline*}
\dual{u'_{0}}{v-u'_{0} }_{\HH^{1}_{\mathrm{D}}(\Omega_0,\R^2)}\geq \int_{\Gamma_{\mathrm{T}_0}}\left(\theta \cdot \nn\left(\partial_{\nn}\left(\mathrm{A}\mathrm{e}(u_0)\nn\right)-\partial_{\nn}\left(\mathrm{A}\mathrm{e}(u_0)\right)\nn\right)+\mathrm{A}\mathrm{e}(u_0)\nabla_{\tau}\left(\theta\cdot\nn\right)\right)\cdot \left(v-u'_{0}\right)\\
+\int_{\Gamma_{\mathrm{T}_0}}\left(-\nabla{\left(\mathrm{A}\mathrm{e}(u_0)\nn\right)}\theta+\left(\nabla{\theta}-\mathrm{div}_{\tau}(\theta)\mathrm{I}\right)\mathrm{A}\mathrm{e}(u_0)\nn\right)\cdot \left(v-u'_{0}\right)-\int_{\Gamma_{\mathrm{T_0}^{u_0,g}_{\mathrm{N}}}}p(\theta)\frac{u_{0_\tau}}{\left\|u_{0_\tau}\right\|}\cdot \left(v_\tau-{u'_{0}}_\tau\right)\\+\int_{\Gamma_{\mathrm{T_0}^{u_0,g}_{\mathrm{S}}}}p(\theta)\frac{\sigma_{\tau}(u_0)}{g}\cdot \left(v_\tau-{u'_{0}}_\tau\right),
\end{multline*}
for all $v\in\mathcal{K}_0-\nabla{\theta}^{\top}u_0-\nabla{u_0}\theta$, which corresponds to the weak formulation of the expected tangential Signorini problem (see~\cite[Section~2.1.2]{BCJDC2} for details).
\end{proof}

\begin{myRem}\label{remdirectional} \normalfont
    Note that $\overline{u}'_0$ and $u'_0$ are not linear with respect to the direction~$\theta$. This nonlinearity is standard in shape optimization for variational inequalities (see, e.g.,~\cite{ABCJ,HINTERMULLERLAURAIN} or~\cite[Section 4]{SOKOZOL}), and justifies the names of \textit{directional} material and shape derivatives.
\end{myRem}

\subsection{Shape gradient of the Tresca energy functional}\label{energyfunctional}

Thanks to the characterization of the directional material and shape derivatives obtained in the previous section, we are now in a position to derive an expression of the shape gradient of the Tresca energy functional~$\mathcal{J}$ at $\Omega_{0}$ in the direction~$\theta$.

\begin{myTheorem}\label{shapederivofJTresca1}
Consider the framework of Proposition~\ref{deriv11}. Then the Tresca energy functional~$\mathcal{J}$ admits a shape gradient at $\Omega_{0}$ in the direction~$\theta$ given by
\begin{multline*}
    \mathcal{J}'(\Omega_{0})(\theta)=\int_{\Omega_0}\mathrm{div}\left(\theta\right)\frac{\mathrm{A}\mathrm{e}\left(u_0\right):\mathrm{e}\left(u_0\right)}{2}-\int_{\Omega_0}\mathrm{div}\left(\mathrm{A}\mathrm{e}\left(u_0\right)\right)\cdot\nabla{u_0}\theta-\int_{\Omega_0}\mathrm{A}\mathrm{e}\left(u_0\right):\nabla{u_0}\nabla{\theta}\\-\int_{\Gamma_{\mathrm{T}_0}}\theta\cdot\nn \left(f\cdot u_0\right)-\dual{\mathrm{A}\mathrm{e}(u_0)\nn}{\nabla{\theta}^{\top}u_0}_{\HH^{-1/2}(\Gamma_0,\R^2)\times\HH^{1/2}(\Gamma_0,\R^2)}\\+\int_{\Gamma_{\mathrm{T_0}^{u_0,g}_{\mathrm{N}}}}\left(\nabla{g}\cdot\theta+g\left( \mathrm{div}_{\tau}(\theta)-\nabla{\theta}\tau\cdot\tau\right)\right)\left\|u_{0_\tau}\right\|.
\end{multline*}
\end{myTheorem}
\begin{proof}
    By taking~$v=u_t$ in the variational inequality satisfied by~$u_t$, one can obtain that
    $$ \mathcal{J}(\Omega_t) = - \dfrac{1}{2} \int_{ \Omega_t } \mathrm{A} \mathrm{e}(u_t) : \mathrm{e}(u_t). $$
    Following the usual strategy developed in (smooth) shape optimization literature (see, e.g.,~\cite{ALL,HENROT}) to compute the shape gradient of $\mathcal{J}$ at $\Omega_{0}$ in the direction $\theta$, one gets
\begin{equation*}
\mathcal{J}'(\Omega_{0})(\theta)=-\frac{1}{2}\int_{\Omega_{0}}\mathrm{div}(\theta)\mathrm{A}\mathrm{e}(u_0):\mathrm{e}(u_0)+\int_{\Omega_{0}}\mathrm{A}\mathrm{e}(u_0):\nabla{u_{0}}\nabla{\theta}-\dual{\overline{u}'_0}{u_{0}}_{_{\HH^{1}_{\mathrm{D}}(\Omega_0,\R^2)}}.
\end{equation*}
Moreover, from the variational inequality satisfied by $\overline{u}'_0$ (see~\eqref{inequalityofmaterialderiv1}), the divergence formula (see Proposition~\ref{div}) and since~$\overline{u}'_0\pm u_0\in\mathcal{K}_0-\nabla{\theta}^{\top}u_0$, it follows that
\begin{multline*}
    \dual{\overline{u}'_0}{u_0 }_{\HH^{1}_{\mathrm{D}}(\Omega_0,\R^2)}=\int_{\Omega_0}\left(\mathrm{div}\left(\mathrm{A}\mathrm{e}(u_0)\right)\theta^{\top}+\left(\mathrm{A}\mathrm{e}(u_0)\right)\nabla{\theta}^{\top}+\mathrm{A}\left(\nabla{u_0}\nabla{\theta}\right)-\mathrm{div}(\theta)\mathrm{A}\mathrm{e}(u_0)\right):\nabla{u_0}\\+\int_{\Gamma_{\mathrm{T}_0}}\theta\cdot\nn (f\cdot u_0)+\dual{\mathrm{A}\mathrm{e}(u_0)\nn}{\nabla{\theta}^{\top}u_0}_{\HH^{-1/2}(\Gamma_0,\R^2)\times\HH^{1/2}(\Gamma_0,\R^2)}-\int_{\Gamma_{\mathrm{T_0}^{u_0,g}_{\mathrm{N}}}}p(\theta)\left\|u_{0_\tau}\right\|.
\end{multline*}
Then, since ${u_0}_\tau=0$ \textit{a.e.}\ on $\Gamma_{\mathrm{T_0}^{u_0,g}_{\mathrm{S}}}$ and using the equality $\mathrm{div}\left(\mathrm{A}\mathrm{e}(u_0)\right)\theta^{\top}:\nabla{u_0}=\mathrm{div}\left(\mathrm{A}\mathrm{e}(u_0)\right)\cdot\nabla{u_0}\theta$ which holds true \textit{a.e.}\ on $\Omega_0$, we conclude the proof.
\end{proof}

As we did for the directional material derivative, the presentation of Theorem~\ref{shapederivofJTresca1} can be improved under additional assumptions.

\begin{myCor}\label{shapederivofJ}
Consider the framework of Proposition~\ref{deriv11} with the additional assumptions that~$u_0\in\HH^3(\Omega_0,\R^2)$, $\Gamma_0$ is of class $\mathcal{C}^3$ and almost every point of $\Gamma_{\mathrm{T_0}}$ belongs to the relative interior $\mathrm{int}_{\Gamma_0}({\Gamma_{\mathrm{T_0}}})$. Then the Tresca energy functional $\mathcal{J}$ admits a shape gradient at $\Omega_{0}$ in the direction~$\theta$ given by
\begin{multline*}
\mathcal{J}'(\Omega_{0})(\theta)=\int_{\Gamma_{\mathrm{T}_{0}}} \theta\cdot\nn\left(\frac{\mathrm{A}\mathrm{e}(u_0):\mathrm{e}(u_0)}{2}-f\cdot u_{0}-\sigma_{\tau}(u_0)\cdot\partial_{\nn}(u_0)+\left\|{u_0}_\tau \right\|\left(Hg+\partial_{\nn}g\right)\right)\\+\int_{\Gamma_{\mathrm{T_0}}}{u_0}_\nn\sigma_\tau (u_0) \cdot\tau\left(\nabla{\tau}\theta_{\tau}-\nabla{\theta}\tau\right)\cdot\nn,
\end{multline*}
where $H$ is the mean curvature of $\Gamma_{0}$.
\end{myCor}

\begin{proof}
Since $u_0\in\HH^2(\Omega_0,\R^2)$, it follows from Theorem~\ref{shapederivofJTresca1} that 
\begin{multline*}
    \mathcal{J}'(\Omega_{0})(\theta)=-\frac{1}{2}\int_{\Omega_{0}}\theta\cdot\nabla{{\left(\mathrm{A}\mathrm{e}(u_0):\mathrm{e}(u_0)\right)}}+\int_{\Gamma_{0}}\theta\cdot\nn\frac{\mathrm{A}\mathrm{e}(u_0):\mathrm{e}(u_0)}{2}+\int_{\Omega_0}\mathrm{A}\mathrm{e}\left(u_0\right):\mathrm{e}\left(\nabla{u_0}\theta\right)\\-\int_{\Gamma_0}\mathrm{A}\mathrm{e}\left(u_0\right)\nn\cdot\nabla{u_0}\theta-\int_{\Omega_0}\mathrm{A}\mathrm{e}\left(u_0\right):\nabla{u_0}\nabla{\theta}-\int_{\Gamma_{\mathrm{T}_0}}\theta\cdot\nn \left(f\cdot u_0\right)-\int_{\Gamma_{\mathrm{T}_0}}\mathrm{A}\mathrm{e}(u_0)\nn\cdot\nabla{\theta}^{\top}u_0\\+\int_{\Gamma_{\mathrm{T_0}^{u_0,g}_{\mathrm{N}}}}\left(\nabla{g}\cdot\theta+g\left( \mathrm{div}_{\tau}(\theta)-\nabla{\theta}\tau\cdot\tau\right)\right)\left\|u_{0_\tau}\right\|.
\end{multline*}
Moreover, since
$$
\mathrm{A}\mathrm{e}(u_0):\mathrm{e}\left(\nabla{u_0}\theta\right)=\mathrm{A}\mathrm{e}(u_0):\nabla{u_0}\nabla{\theta}+\frac{1}{2}\theta\cdot\nabla{{\left(\mathrm{A}\mathrm{e}(u_0):\mathrm{e}(u_0)\right)}} \text{ \textit{a.e.}\ on } \Omega_0,
$$
one deduces that
\begin{multline*}
     \mathcal{J}'(\Omega_{0})(\theta)=\int_{\Gamma_{0}}\theta\cdot\nn\left(\frac{\mathrm{A}\mathrm{e}(u_0):\mathrm{e}(u_0)}{2}\right)-\int_{\Gamma_0}\mathrm{A}\mathrm{e}(u_0)\nn\cdot\nabla{u_0}\theta-\int_{\Gamma_{\mathrm{T}_0}}\theta\cdot\nn (f\cdot u_0)-\int_{\Gamma_0}\mathrm{A}\mathrm{e}(u_0)\nn\cdot\nabla{\theta}^{\top}u_0\\+\int_{\Gamma_{\mathrm{T_0}^{u_0,g}_{\mathrm{N}}}}\left(\nabla{g}\cdot\theta+g\left( \mathrm{div}_{\tau}(\theta)-\nabla{\theta}\tau\cdot\tau\right)\right)\left\|u_{0_\tau}\right\|.
\end{multline*}
Furthermore, since almost every point of $\Gamma_{\mathrm{T_0}}$ belongs to $\mathrm{int}_{\Gamma_0}({\Gamma_{\mathrm{T_0}}})$,
it follows that $u_0$ is a strong solution to the Tresca friction problem (see~\cite[Definition~2.11 and Proposition~2.13]{BCJDC2}). Thus, from the Tresca friction law, one has $\left\|u_{0_\tau}\right\|=-\frac{\sigma_\tau (u_0) \cdot {u_0}_\tau}{g}$ \textit{a.e.}\ on $\Gamma_{\mathrm{T_0}}$ and, since~${u_0}_\tau=0$ on~$\Gamma_{\mathrm{T_0}^{u_0,g}_{\mathrm{D}}}\cup\Gamma_{\mathrm{T_0}^{u_0,g}_{\mathrm{S}}}$ and $\theta=0$ on $\Gamma_{\mathrm{D}}$, one gets that
\begin{multline}\label{equalityH2}
     \mathcal{J}'(\Omega_{0})(\theta)=\int_{\Gamma_{0}}\theta\cdot\nn\left(\frac{\mathrm{A}\mathrm{e}(u_0):\mathrm{e}(u_0)}{2}-f\cdot u_0\right)-\int_{\Gamma_0}\mathrm{A}\mathrm{e}(u_0)\nn\cdot\nabla{u_0}\theta-\int_{\Gamma_0}\mathrm{A}\mathrm{e}(u_0)\nn\cdot\nabla{\theta}^{\top}u_0\\-\int_{\Gamma_{\mathrm{0}}}\left(g\sigma_\tau  (u_0) \cdot {u_0}_\tau\frac{\nabla{g}}{g^2}\cdot\theta+\sigma_\tau (u_0) \cdot {u_0}_\tau  \mathrm{div}_{\tau}(\theta)\right)+\int_{\Gamma_{\mathrm{0}}}\sigma_\tau (u_0)\cdot {u_0}_\tau\nabla{\theta}\tau\cdot\tau .
\end{multline}
Moreover, since $\Gamma_0$ is of class $\mathcal{C}^{3}$, $\nn\in\mathcal{C}^2(\Gamma_0,\R^2)$ can be extended over~$\R^2$ such that~$\nn\in\mathcal{C}^2(\R^2,\R^2)$ and $||\nn||=1$ over $\R^2$ (see, e.g.,~\cite[Chapter 5 Section 5.4]{HENROT}). It follows that~$\tau\in\mathcal{C}^2(\Gamma_0,\R^2)$ can also be extended over~$\R^2$ such that~$\tau\in\mathcal{C}^2(\R^2,\R^2)$ and~$||\tau||=1$ over $\R^2$. Moreover, since~$u_0\in\HH^3(\Omega_0,\R^2)$, the shear stress $\sigma_\tau (u_0)=\mathrm{A}\mathrm{e}(u_0)\nn - (\mathrm{A}\mathrm{e}(u_0)\nn\cdot \nn) \nn$ can be extended into a function defined over $\Omega_{0}$ such that~$\sigma_\tau(u_0)\in\HH^{2}(\Omega_{0},\R^2)$. Thus~$\sigma_\tau(u_0)\cdot\tau\in\HH^{2}(\Omega_{0},\R^2)$, ${u_0}\cdot\tau\in\HH^{2}(\Omega_{0},\R^2)$ and one deduces that
\begin{multline*}
     \mathcal{J}'(\Omega_{0})(\theta)=\int_{\Gamma_{0}}\theta\cdot\nn\left(\frac{\mathrm{A}\mathrm{e}(u_0):\mathrm{e}(u_0)}{2}-f\cdot u_0\right)-\int_{\Gamma_0}\mathrm{A}\mathrm{e}(u_0)\nn\cdot\nabla{u_0}\theta-\int_{\Gamma_0}\mathrm{A}\mathrm{e}(u_0)\nn\cdot\nabla{\theta}^{\top}u_0\\-\int_{\Gamma_{0}}\left(\nabla{\left(\left(\sigma_\tau(u_0)\cdot\tau\right){u_0}\cdot\tau\right)}\cdot\theta+\sigma_\tau (u_0) \cdot {u_0}_\tau  \mathrm{div}_{\tau}(\theta)\right)+\int_{\Gamma_{0}}\sigma_\tau (u_0) \cdot\tau \nabla{\left(u_0\cdot\tau\right)\cdot\theta}\\+\int_{\Gamma_{0}}g{u_0}\cdot\tau\nabla{\left(\frac{\sigma_\tau(u_0)\cdot\tau}{g}\right)}\cdot\theta+\int_{\Gamma_{0}}\sigma_\tau (u_0) \cdot {u_0}_\tau\nabla{\theta}\tau\cdot\tau,
\end{multline*}
and, since~$\left(\sigma_\tau(u_0)\cdot\tau\right)u_0\cdot\tau\in\mathrm{W}^{1,2}(\Omega_0,\R)$, one can apply Proposition~\ref{intbord} to get that
\begin{multline*}
     \mathcal{J}'(\Omega_{0})(\theta)=\int_{\Gamma_{0}}\theta\cdot\nn\left(\frac{\mathrm{A}\mathrm{e}(u_0):\mathrm{e}(u_0)}{2}-f\cdot u_0-H\sigma_\tau (u_0)\cdot {u_0}_\tau-\partial_{\nn}\left(\left(\sigma_{\tau}(u_0)\cdot\tau\right)u_0\cdot\tau\right)\right)\\-\int_{\Gamma_0}\mathrm{A}\mathrm{e}(u_0)\nn\cdot\nabla{u_0}\theta-\int_{\Gamma_0}\mathrm{A}\mathrm{e}(u_0)\nn\cdot\nabla{\theta}^{\top}u_0+\int_{\Gamma_{0}}\sigma_\tau (u_0) \cdot\tau \nabla{\left(u_0\cdot\tau\right)\cdot\theta}\\+\int_{\Gamma_{0}}g{u_0}\cdot\tau\nabla{\left(\frac{\sigma_\tau(u_0)\cdot\tau}{g}\right)}\cdot\theta+\int_{\Gamma_{0}}\sigma_\tau (u_0) \cdot {u_0}_\tau\nabla{\theta}\tau\cdot\tau.
\end{multline*}
Moreover, since~$\sigma_{\nn}(u_0)=0$ \textit{a.e.}\ on $\Gamma_{\mathrm{T_0}}$, note that
\begin{multline*}
    -\int_{\Gamma_0}\mathrm{A}\mathrm{e}(u_0)\nn\cdot\nabla{\theta}^{\top}u_0+\int_{\Gamma_{0}}\sigma_\tau (u_0)\cdot {u_0}_\tau\nabla{\theta}\tau\cdot\tau=-\int_{\Gamma_{\mathrm{T_0}}}\left(\sigma_\tau(u_0)\cdot\tau\right)\nabla{\theta}\tau\cdot u_0\\+\int_{\Gamma_{\mathrm{T_0}}}\sigma_\tau (u_0) \cdot {u_0}_\tau\nabla{\theta}\tau\cdot\tau=-\int_{\Gamma_{\mathrm{T_0}}}\sigma_\tau(u_0) \cdot {u_0}_\tau \nabla{\theta}\tau\cdot\tau -\int_{\Gamma_{\mathrm{T_0}}}\left(\sigma_\tau (u_0) \cdot \tau\right) {u_0}_\nn\nabla{\theta}\tau\cdot\nn\\+\int_{\Gamma_{\mathrm{T_0}}}\sigma_\tau (u_0) \cdot {u_0}_\tau\nabla{\theta}\tau\cdot\tau=-\int_{\Gamma_{\mathrm{T_0}}}\left(\sigma_\tau (u_0)\cdot \tau\right) {u_0}_\nn\nabla{\theta}\tau\cdot\nn,
\end{multline*}
and also that
\begin{multline*}
\int_{\Gamma_{0}}\theta\cdot\nn\left(-\partial_{\nn}\left(\left(\sigma_\tau(u_0)\cdot\tau\right)u_0\cdot\tau\right)\right)-\int_{\Gamma_0}\mathrm{A}\mathrm{e}(u_0)\nn\cdot\nabla{u_0}\theta+\int_{\Gamma_{0}}\sigma_\tau (u_0) \cdot\tau \nabla{\left(u_0\cdot\tau\right)\cdot\theta}\\ = \int_{\Gamma_{\mathrm{T_0}}}\theta\cdot\nn\left(-\partial_{\nn}\left(u_0\cdot\tau\right)\sigma_\tau(u_0)\cdot\tau-\partial_{\nn}\left(\sigma_\tau(u_0)\cdot\tau\right)u_0\cdot\tau\right)-\int_{\Gamma_{\mathrm{T_0}}}\left(\sigma_\tau (u_0) \cdot\tau\right)\tau\cdot\nabla{u_0}\theta\\+\int_{\Gamma_{\mathrm{T_0}}}\theta\cdot\nn\left(\sigma_\tau (u_0) \cdot\tau\right) \nabla{\left(u_0\cdot\tau\right)\cdot\nn}+\int_{\Gamma_{\mathrm{T_0}}}\left(\sigma_\tau (u_0) \cdot\tau\right) \nabla{\left(u_0\cdot\tau\right)\cdot\theta_{\tau}}\\=\int_{\Gamma_{\mathrm{T_0}}}\theta\cdot\nn\left(-\partial_{\nn}\left(\sigma_\tau(u_0)\cdot\tau\right)u_0\cdot\tau\right)-\int_{\Gamma_{\mathrm{T_0}}}\theta\cdot\nn\left(\sigma_\tau (u_0) \cdot\tau\right)\tau\cdot\nabla{u_0}\nn-\int_{\Gamma_{\mathrm{T_0}}}\left(\sigma_\tau (u_0) \cdot\tau\right)\theta_{\tau}\cdot\nabla{u_0}\tau\\+\int_{\Gamma_{\mathrm{T_0}}}\left(\sigma_\tau (u_0) \cdot\tau\right) {\nabla{u_0}}^{\top}\tau\cdot\theta_{\tau}+\int_{\Gamma_{\mathrm{T_0}}}\left(\sigma_\tau (u_0) \cdot\tau\right) \nabla{\tau}^{\top}u_0\cdot\theta_{\tau}\\=-\int_{\Gamma_{\mathrm{T_0}}}\theta\cdot\nn\left(\partial_{\nn}\left(\sigma_\tau(u_0)\cdot\tau\right)u_0\cdot\tau+\sigma_\tau (u_0)\cdot\partial_{\nn}(u_0)\right)+\int_{\Gamma_{\mathrm{T_0}}}\left(\sigma_\tau (u_0) \cdot\tau\right) {u_0}_\nn\nabla{\tau}^{\top}\nn\cdot\theta_{\tau},
\end{multline*}
since $||\tau||$=1 on $\R^2$, thus $(\nabla{\tau})\tau\cdot\tau=0$ on $\Gamma_{\mathrm{T_0}}$.
Hence one has
\begin{multline*}
     \mathcal{J}'(\Omega_{0})(\theta)=\int_{\Gamma_{\mathrm{T_0}}}g{u_0}\cdot\tau\nabla{\left(\frac{\sigma_\tau(u_0)\cdot\tau}{g}\right)}\cdot\theta+\int_{\Gamma_{\mathrm{T_0}}}{u_0}_\nn\sigma_\tau (u_0) \cdot\tau\left(\nabla{\tau}\theta_{\tau}-\nabla{\theta}\tau\right)\cdot\nn\\+\int_{\Gamma_{\mathrm{T_0}}}\theta\cdot\nn\left(\frac{\mathrm{A}\mathrm{e}(u_0):\mathrm{e}(u_0)}{2}-f\cdot u_0-H\sigma_\tau (u_0)\cdot {u_0}_\tau-u_0\cdot\tau\partial_{\nn}\left(\sigma_{\tau}(u_0)\cdot\tau\right)-\sigma_{\tau}(u_0)\cdot\partial_{\nn}(u_0)\right).
\end{multline*}
Now let us focus on the first term. Since  ${u_{0}}_\tau=0$ on $\Gamma_{\mathrm{T_0}^{u_0,g}_{\mathrm{D}}}\cup\Gamma_{\mathrm{T_0}^{u_0,g}_{\mathrm{S}}}$, we have
$$
\int_{\Gamma_{\mathrm{T_0}}}g{u_0}\cdot\tau\nabla{\left(\frac{\sigma_\tau(u_0)\cdot\tau}{g}\right)}\cdot\theta=\int_{\Gamma_{\mathrm{T_0}^{u_0,g}_{\mathrm{N}}}}g{u_0}\cdot\tau\nabla{\left(\frac{\sigma_\tau(u_0)\cdot\tau}{g}\right)}\cdot\theta.
$$ 
Let us introduce two disjoint subsets of $\Gamma_{\mathrm{T_0}^{u_0,g}_{\mathrm{N}}}$ given by
$$
\Gamma_{\mathrm{T_0}^{u_0,g}_{\mathrm{N+}}}:=\left\{s\in\Gamma_{\mathrm{T_0}} \mid {u_{0}}(s)\cdot\tau(s)>0\right\}
\quad \mbox{ and } \quad 
\Gamma_{\mathrm{T_0}^{u_0,g}_{\mathrm{N-}}}:=\left\{s\in\Gamma_{\mathrm{T_0}} \mid {u_{0}}(s)\cdot\tau(s)<0\right\}.
$$
It follows that $\Gamma_{\mathrm{T_0}^{u_0,g}_{\mathrm{N}}}=\Gamma_{\mathrm{T_0}^{u_0,g}_{\mathrm{N+}}}\cup\Gamma_{\mathrm{T_0}^{u_0,g}_{\mathrm{N-}}}$, with $\sigma_\tau (u_0)\cdot\tau=-g$ \textit{a.e.}\ on $\Gamma_{\mathrm{T_0}^{u_0,g}_{\mathrm{N+}}}$, and~$\sigma_\tau (u_0)\cdot\tau=g$ \textit{a.e.}\ on $\Gamma_{\mathrm{T_0}^{u_0,g}_{\mathrm{N-}}}$. Moreover, since $u_{0}\in\HH^{3}(\Omega,\R^2)$, we get from Sobolev embeddings (see, e.g.,~\cite[Chapter~4]{ADAMS}) that $u_{0}$ is continuous over $\Gamma_{\mathrm{T_0}}$, thus $\Gamma_{\mathrm{T_0}^{u_0,g}_{\mathrm{N+}}}$ and $\Gamma_{\mathrm{T_0}^{u_0,g}_{\mathrm{N-}}}$ are open subsets of~$\Gamma_{\mathrm{T_0}}$. Hence $\nabla_{\tau}(\frac{\sigma_\tau (u_0)\cdot \tau}{g})=0$ \textit{a.e.}\ on $\Gamma_{\mathrm{T_0}^{u_0,g}_{\mathrm{N+}}}\cup\Gamma_{\mathrm{T_0}^{u_0,g}_{\mathrm{N-}}}$, and one deduces that 
$$
\int_{\Gamma_{\mathrm{T_0}}}g{u_0}\cdot\tau\nabla{\left(\frac{\sigma_\tau(u_0)\cdot\tau}{g}\right)}\cdot\theta=\int_{\Gamma_{\mathrm{T_0}}}\theta\cdot\nn\left(g{u_0}\cdot\tau\nabla{\left(\frac{\sigma_\tau(u_0)\cdot\tau}{g}\right)}\cdot\nn\right), 
$$
that is
$$
\int_{\Gamma_{\mathrm{T_0}}}g{u_0}\cdot\tau\nabla{\left(\frac{\sigma_\tau(u_0)\cdot\tau}{g}\right)}\cdot\theta=\int_{\Gamma_{\mathrm{T_0}}}\theta\cdot\nn\left({u_0}\cdot\tau\partial_{\nn}\left(\sigma_\tau (u_0)\cdot\tau\right)-\frac{\sigma_\tau (u_0) \cdot {u_0}_\tau}{g}\partial_{\nn}g\right).
$$
Then, using the Tresca friction law, one has $\sigma_\tau (u_0)\cdot {u_0}_\tau=-g||{u_0}_\tau||$ \textit{a.e.}\ on $\Gamma_{\mathrm{T_0}}$, which concludes the proof.
\end{proof}

\begin{myRem}\normalfont
Under the weaker condition~$u_0\in\HH^2(\Omega_0,\R^2)$, one can follow the proof of Corollary~\ref{shapederivofJ} and obtain that the shape gradient of $\mathcal{J}$ is given by Equality~\eqref{equalityH2}.
\end{myRem}

\section{Numerical illustration}\label{numericalsim}
In this section our objective is to numerically solve a toy example of the shape optimization problem~\eqref{shapeOptim}, by making use of the theoretical results established in this work. The numerical simulations have been performed using Freefem++ software~\cite{HECHT} with P1-finite elements and standard affine mesh. We could use the expression of the shape gradient of~$\mathcal{J}$ obtained in Theorem~\ref{shapederivofJTresca1} but, in order to simplify the computations, we chose to use the expression provided in Corollary~\ref{shapederivofJ} under additional assumptions that we assumed to be true at each iteration.

\subsection{Numerical methodology}\label{methodnum}
Consider an initial shape~$\Omega_0 \in \mathcal{U}_{\mathrm{ref}}$. Note that Corollary~\ref{shapederivofJ} allows to exhibit a descent direction~$\theta_0$ of the Tresca energy functional~$\mathcal{J}$ at~$\Omega_0$, by finding the unique solution $\theta_0 \in\HH^{1}_{\mathrm{D}}(\Omega_0,\R^2)$ to the variational equality
\begin{equation*}
\int_{\Omega_0}\left(\nabla{\theta_0}:\nabla\theta+\theta_0\cdot\theta\right)=-\mathcal{J}'(\Omega_0)(\theta), \qquad \forall\theta\in\HH^{1}_{\mathrm{D}}(\Omega_0,\R^2),
\end{equation*}
since it satisfies~$\mathcal{J}'(\Omega_{0})(\theta_0)=-\int_{\Omega_0}\left(||\nabla{\theta_0}||^2+||\theta_0||^2\right)\leq0
$.

In order to numerically solve the shape optimization problem~\eqref{shapeOptim} on a given example, we have to deal with the volume constraint~$\vert \Omega \vert = \vert \Omega_{\mathrm{ref}} \vert>0$. For this purpose, the Uzawa algorithm (see, e.g.,~\cite[Chapter 3]{ALL}) is used, and one refers to~\cite[Section 4]{ABCJ} for methodological details. 

Let us mention that the Tresca friction problem is numerically solved using an adaptation of iterative switching algorithms (see~\cite{11AIT}). This algorithm operates by checking at each iteration if the Tresca boundary conditions are satisfied and, if they are not, by imposing them and restarting the computation (see~\cite[Appendix C p.25]{4ABC} for detailed explanations). We also precise that, for all~$j\in~\mathbb{N}^{*}$, the difference between the Tresca energy functional $\mathcal{J}$ at the iteration $20\times j$ and at the iteration~$20\times (j-1)$ is computed. The smallness of this difference is used as a stopping criterion for the algorithm. Finally the curvature term~$H$ is numerically computed by extending the normal~$\mathrm{n}$ into a function~$\tilde{\mathrm{n}}$ which is defined on the whole domain~$\Omega_0$. Then the curvature is given by~$H=\mathrm{div}(\tilde{\mathrm{n}})-\nabla(\tilde{\mathrm{n}}) \mathrm{n} \cdot \mathrm{n}$ (see, e.g.,~\cite[Proposition 5.4.8]{HENROT}).

\subsection{A toy example and numerical results}
In this section, let $f\in\HH^{1}(\R^{2},\R^2)$ be defined by
$$
\fonction{f}{\R^{2}}{\R^2}{(x,y)}{\displaystyle f(x,y):= \begin{pmatrix}
-5x\exp{(x)} & 0.6\exp{(x^2)}
\end{pmatrix}\eta(x,y),}
$$
and $g\in\HH^2(\R^2,\R)$ be defined by
$$
\fonction{g}{\R^{2}}{\R}{(x,y)}{ g(x,y) := \left(1+\sin(-y\frac{\pi}{2})+10^{-3}\right)\eta(x,y),}
$$
where $\eta \in\mathcal{C}^{\infty}(\R^2,\R)$ is a  cut-off function chosen appropriately so that $f$ belongs to~$\HH^{1}(\R^{2},\R^2)$, $g\in\HH^2(\R^2,\R)$ and $g>0$ on $\R^2$. The reference shape $\Omega_{\mathrm{ref}}\subset\R^{2}$ is an ellipse centered at~$(0,0)\in\R^2$, with semi-major axis $a=1.1$ and semi-minor axis $b=1/a$, and the fixed part $\Gamma_{\mathrm{D}}$ is given by
$$
\Gamma_{\mathrm{D}} := \left\{\left(a\cos{\gamma},b\sin{\gamma}\right)\in\Gamma_{\mathrm{ref}} \mid \gamma\in\left[\frac{2\pi}{3},\frac{4\pi}{3}\right]\cup\left[\frac{5\pi}{3},\frac{7\pi}{3}\right] \right \}.
$$
We refer to Figure~\ref{figurrre}. The volume constraint is $\vert \Omega_{\mathrm{ref}} \vert=\pi$
 and the initial shape is $\Omega_{0}:=\Omega_{\mathrm{ref}}$.
 
In the sequel we consider that, for all $\Omega\in\mathcal{U}_{\mathrm{ref}}$, the Cauchy stress tensor~$\sigma$, defined by~$\sigma(v):=\mathrm{A}\mathrm{e}(v)$ for all $v\in\HH^{1}_{\mathrm{D}}(\Omega,\R^2)$, satisfies
 $$
 \sigma(v)=2\mu\mathrm{e}(v)+\lambda \mathrm{tr}\left(\mathrm{e}(v)\right)\mathrm{I},
 $$
for all $v\in\HH^{1}_{\mathrm{D}}(\Omega,\R^2)$, where $\mathrm{tr}\left(\mathrm{e}(v)\right)$ is the trace of the matrix $\mathrm{e}(v)$, and $\mu\geq0,\lambda\geq0$ are Lamé parameters (see, e.g.,~\cite{SALEN}). From a physical point of view, this assumption corresponds to \textit{isotropic} elastic solids. In the sequel we consider the arbitrary data $\mu=0.5$ and~$\lambda=0$.
 
 We present now the numerical results obtained for this toy example using the numerical methodology described in Section~\ref{methodnum}.

 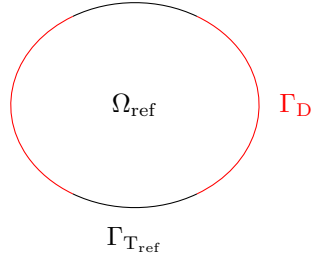
\begin{figure}[ht]
  \centering
\begin{tikzpicture}[scale=1.5]
\draw (0,0) node{$\Omega_{\mathrm{ref}}$};
\draw  [color=black] (0.55,0.7872) arc (60:120:1.1 and 0.909);
\draw  [color=red] (-0.55,0.7872) arc (120:240:1.1 and 0.909);
\draw  [color=black] (-0.55,-0.7872) arc (240:300:1.1 and 0.909);
\draw  [color=red] (0.55,-0.7872) arc (300:420:1.1 and 0.909);
\draw (1.2,0) [color=red] node[right]{$\Gamma_{\mathrm{D}}$};
\draw (0,-1) [color=black] node[below]{$\Gamma_{\mathrm{T}_{\mathrm{ref}}}$};
  \end{tikzpicture}
  \caption{$\Omega_{\mathrm{ref}}$ and its boundary $\Gamma_{\mathrm{ref}}=\Gamma_{\mathrm{D}}\cup\Gamma_{\mathrm{T}_{\mathrm{ref}}}$.}\label{figurrre}
  \end{figure}

 In Figure~\ref{fig1} is represented the initial shape (left) and the shape which numerically solves Problem~\eqref{shapeOptim} (right). On top are the vector values of the solution $u$ to the Tresca friction problem~\eqref{Trescaproblem2222}. On the initial shape, note that, on the bottom blue boundary, the norm of the shear stress is strictly inferior at the friction threshold $g$, thus $u_\tau=0$, while the top black boundary shows some points where the norm of the shear stress reaches the friction threshold.

 
\begin{figure}[h!]
    \centering
    \includegraphics[scale=0.3]{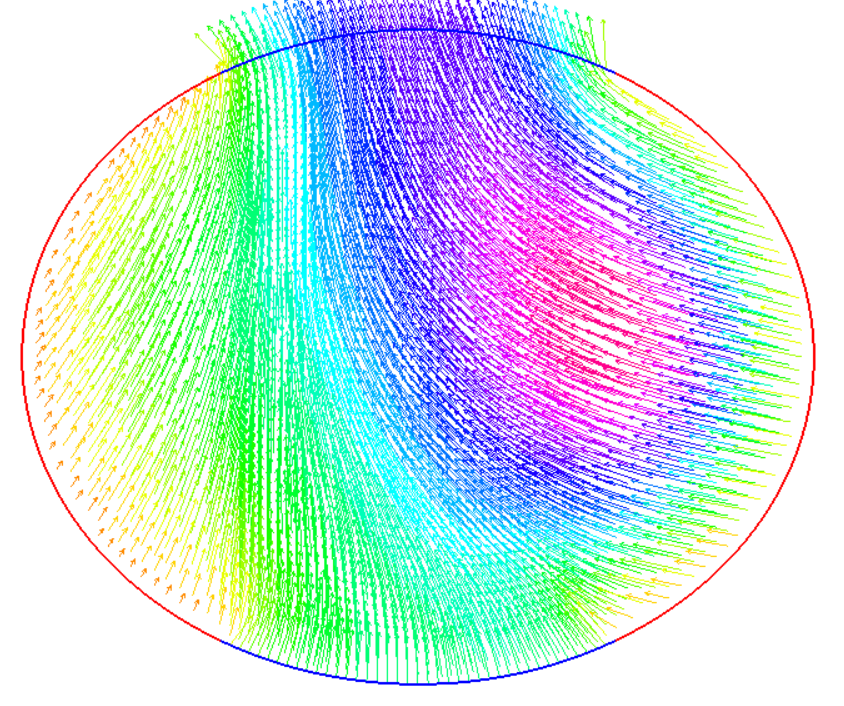}
    \includegraphics[scale=0.3]{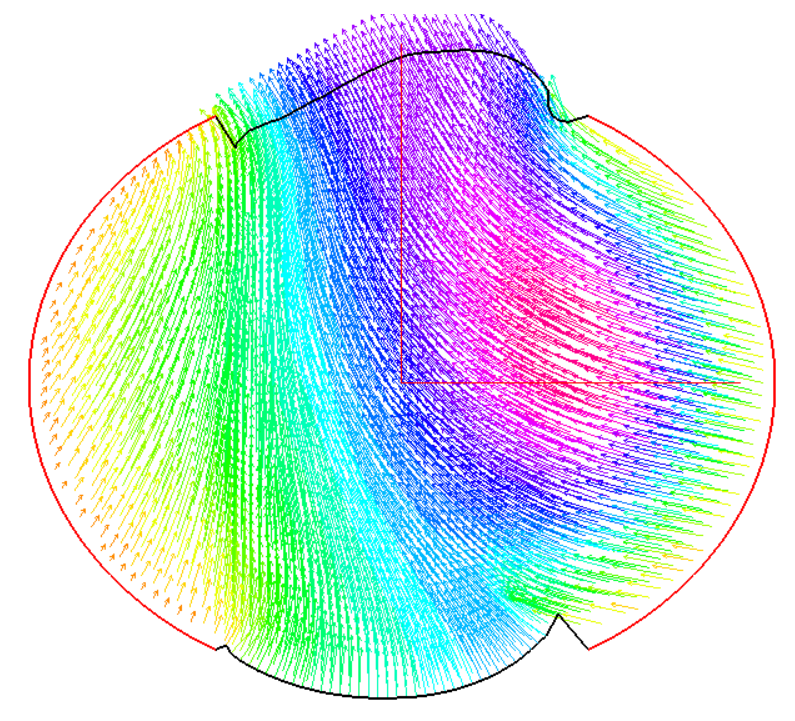}  
    \caption{Initial shape (left) and the shape minimizing $\mathcal{J}$ under the volume constraint~$\vert \Omega \vert=\pi$ (right) . }
    \label{fig1}
\end{figure}

 Figure~\ref{figure3} shows the values of $\mathcal{J}$ (left) and the volume~$\vert \Omega \vert$ of the shape (right) with respect to the iterations. We observe that $\mathcal{J}$ is lower at the final shape, than at the initial shape, with some oscillations due to the Lagrange multiplier in order to satisfy the volume constraint. 

\begin{figure}[h!]
    \centering
\includegraphics[scale=0.45]{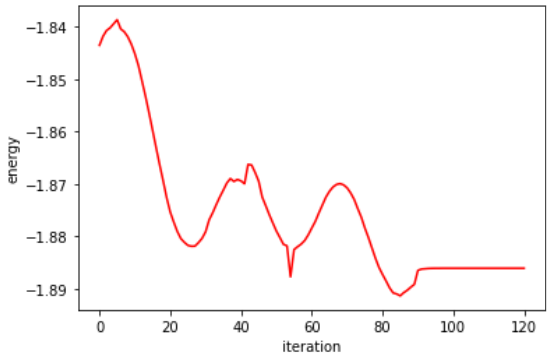}
    \hspace{2cm}
    \includegraphics[scale=0.45]{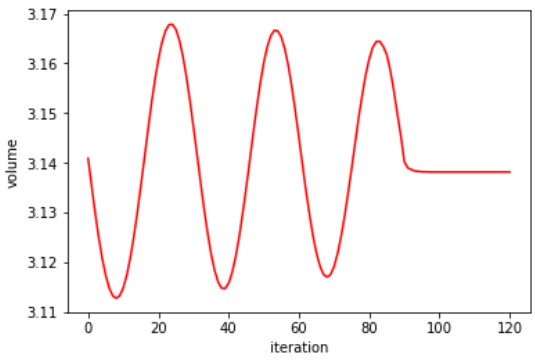}  
    \caption{Values of the energy functional (left) and of the volume (right) with respect to the iterations.}
    \label{figure3}
\end{figure}

\appendix

\section{Reminders on twice epi-differentiability}\label{appendix}

For notions and results recalled in this appendix, we refer to standard references from nonsmooth analysis literature such as~\cite{BREZ2,MINTY,ROCK2} and~\cite[Chapter~12]{ROCK}. In what follows,~$(\mathcal{H}, \dual{\cdot}{\cdot}_{\mathcal{H}})$ stands for a general real Hilbert space. The \textit{domain} and the \textit{epigraph} of an extended real-valued function~$\psi : \mathcal{H}\rightarrow \mathbb{R}\cup\left\{\pm \infty \right\}$ are respectively defined by
$$
\mathrm{dom}\left(\psi\right):=\left\{x\in \mathcal{H} \mid \psi(x)<+\infty \right\} \quad \text{and} \quad
\mathrm{epi}\left(\psi\right):=\left\{(x,\beta)\in \mathcal{H}\times\mathbb{R}\mid \psi(x)\leq \beta\right\}.
$$
Recall that $\psi$ is said to be \textit{proper} if $\mathrm{dom}(\psi)\neq \emptyset$ and $\psi(x)>-\infty$ 
for all~$x\in\mathcal{H}$, and that $\psi$ is convex (resp.\ lower semi-continuous) if and only if~$\mathrm{epi}(\psi)$ is a convex (resp.\ closed) subset of~$\mathcal{H}\times\R$. When $\psi$ is proper, we denote by $\partial{\psi}  :  \mathcal{H} \rightrightarrows \mathcal{H}$ its \textit{convex subdifferential operator}, defined by 
 $$
 \partial{\psi}(x):=\left\{y\in\mathcal{H} \mid \forall z\in\mathcal{H}\text{, } \dual{y}{z-x}_{\mathcal{H}}\leq \psi(z)-\psi(x)\right\},
 $$
when $x\in \mathrm{dom}(\psi)$, and by $\partial{\psi}(x):= \emptyset$ when $x\notin \mathrm{dom}(\psi)$. The notion of \textit{proximal operator} has been introduced by J.J.\ Moreau in 1965 (see~\cite{MOR}) as follows.

\begin{myDefn}[Proximal operator]\label{proxi}
The \emph{proximal operator} associated with a proper, lower semi-continuous and convex function $\psi  : \mathcal{H} \rightarrow \R\cup\left\{+\infty\right\}$ is the map~$\mathrm{prox}_{\psi}  :  \mathcal{H} \rightarrow \mathcal{H}$ defined by
 $$
      \mathrm{prox}_{\psi}(x):=\underset{y\in\mathcal{H}}{\argmin}\left[ \psi(y)+\frac{1}{2}\left \| y-x \right \|^{2}_{\mathcal{H}}\right]=(\mathrm{id}+\partial \psi)^{-1}(x),
 $$
for all $x\in\mathcal{H}$, where $\mathrm{id}  :  \mathcal{H}\rightarrow \mathcal{H}$ stands for the identity operator.
\end{myDefn}

Recall that, if $\psi  :  \mathcal{H} \rightarrow \R\cup\left\{+\infty\right\}$ is a proper, lower semi-continuous and convex function, then its subdifferential~$\partial{\psi}$ is maximal monotone (see, e.g.,~\cite{ROCK2}), and thus its proximal operator~$\mathrm{prox}_{\psi} : \mathcal{H} \rightarrow \mathcal{H}$ is well-defined, single-valued and nonexpansive, i.e.\ Lipschitz continuous with modulus $1$ (see, e.g.,~\cite[Chapter II]{BREZ2}). 

The notion of \textit{twice epi-differentiability} introduced by R.T.\ Rockafellar in~1985 (see~\cite{Rockafellar}) is defined as the Mosco epi-convergence of second-order difference quotient functions. In what follows, we provide reminders and backgrounds on these notions for the reader's convenience. For more details, we refer to~\cite[Chapter 7, Section B]{ROCK} for the finite-dimensional case and to~\cite{DO} for the infinite-dimensional case. In the sequel, the strong (resp.\ weak) convergence of a sequence in~$\mathcal{H}$ will be denoted by~$\rightarrow$ (resp.\ $\rightharpoonup$) and all limits with respect to~$t$ will be considered for~$t \to 0^+$.

\begin{myDefn}[Mosco convergence]\label{limitemuch}
The \emph{outer}, \emph{weak-outer}, \emph{inner} and \emph{weak-inner limits} of a parameterized family~$(S_{t})_{t>0}$ of subsets of $\mathcal{H}$ are respectively defined by
\begin{eqnarray*}
      \mathrm{lim}\sup S_{t}&:=&\left\{ x\in \mathcal{H} \mid \exists (t_{n})_{n\in\mathbb{N}}\rightarrow 0^{+}, \; \exists \left(x_{n}\right)_{n\in\mathbb{N}}\rightarrow x, \; \forall n\in\mathbb{N}, \; x_{n}\in S_{t_{n}}\right\},\\
     \mathrm{w}\text{-}\mathrm{lim}\sup S_{t}&:=&\left\{ x\in \mathcal{H} \mid \exists (t_{n})_{n\in\mathbb{N}}\rightarrow 0^{+}, \; \exists \left(x_{n}\right)_{n\in\mathbb{N}}\rightharpoonup x, \; \forall n\in\mathbb{N}, \; x_{n}\in S_{t_{n}}\right\},\\
     \mathrm{lim}\inf S_{t}&:=&\left\{ x\in \mathcal{H} \mid \forall (t_{n})_{n\in\mathbb{N}}\rightarrow 0^{+}, \; \exists \left(x_{n}\right)_{n\in\mathbb{N}}\rightarrow x, \; \exists N\in\mathbb{N}, \; \forall n\geq N, \; x_{n}\in S_{t_{n}}\right\},\\
     \mathrm{w}\text{-}\mathrm{lim}\inf S_{t}&:=&\left\{ x\in \mathcal{H} \mid \forall (t_{n})_{n\in\mathbb{N}}\rightarrow 0^{+}, \; \exists \left(x_{n}\right)_{n\in\mathbb{N}}\rightharpoonup x, \; \exists N\in\mathbb{N}, \; \forall n\geq N, \; x_{n}\in S_{t_{n}}\right\}.
\end{eqnarray*}
The family~$(S_{t})_{t>0}$ is said to be \emph{Mosco convergent} if~$
\mathrm{w}\text{-}\mathrm{lim}\sup S_{t}\subset\mathrm{lim}\inf S_{t}
$. In that case, all the previous limits are equal and we write
$$
     \mathrm{M}\text{-}\mathrm{lim}~S_{t}:=\mathrm{lim}\inf S_{t}=\mathrm{lim}\sup S_{t}=\mathrm{w}\text{-}\mathrm{lim}\inf S_{t}=\mathrm{w}\text{-}\mathrm{lim}\sup S_{t}.
$$
\end{myDefn}

\begin{myDefn}[Mosco epi-convergence]\label{Mosco}
  Let $(\psi_{t})_{t>0}$ be a parameterized family of functions~$\psi_{t}  : \mathcal{H}\rightarrow \mathbb{R}\cup\left\{\pm \infty \right\}$ for all $t>0$.
 We say that $(\psi_{t})_{t>0}$ is \emph{Mosco epi-convergent} if~$(\mathrm{epi}(\psi_{t}))_{t>0}$ is Mosco convergent in~$\mathcal{H}\times \R$. Then we denote by $\mathrm{ME}\text{-}\mathrm{lim}~ \psi_{t}  :  \mathcal{H}\rightarrow \mathbb{R}\cup\left\{\pm \infty \right\}$ the function characterized by its epigraph~$\mathrm{epi}\left(\mathrm{ME}\text{-}\mathrm{lim}~\psi_{t}\right):=\mathrm{M}\text{-}\mathrm{lim}$ $\displaystyle \mathrm{epi}\left(\psi_{t}\right)$ and we say that $(\psi_{t})_{t>0}$ \emph{Mosco epi-converges} to~$\mathrm{ME}\text{-}\mathrm{lim}~\psi_{t}$.
 \end{myDefn}
 
The notion of \textit{twice epi-differentiability} was originally introduced in~\cite{Rockafellar} for nonparameterized convex functions. However, the framework of the present paper requires an extended version to parameterized convex functions which has been developed in~\cite{8AB}. 
To provide reminders on this extended notion, when considering a function~$\Psi  :  \mathbb{R}_{+}\times \mathcal{H}\rightarrow \mathbb{R}\cup\left\{+\infty\right\}$ such that, for all $t\geq0$, $\Psi(t,\cdot) : \mathcal{H}\rightarrow \mathbb{R} \cup\left\{+\infty\right\}$ is a proper function, we will make use of the following two notations:~$\partial \Psi(0,\mathord{\cdot} )(x)$ stands for the convex subdifferential operator at~$x\in\mathcal{H}$ of the function~$\Psi(0,\cdot)$, and, for each $t\geq 0$, $\Psi^{-1}(t , \mathbb{R}):=\left\{ x\in\mathcal{H}\mid \; \Psi(t,x)\in\R \right\}$ and $\Psi^{-1}(\mathord{\cdot} , \mathbb{R}):=\displaystyle\cap_{t\geq 0}\Psi^{-1}(t , \mathbb{R})  $.

 \begin{myDefn}[Twice epi-differentiability depending on a parameter]\label{epidiffpara}
Let~$\Psi  :  \mathbb{R}_{+}\times \mathcal{H}\rightarrow \mathbb{R} \cup\left\{+\infty\right\}$ be a function such that, for all $t\geq0$, $\Psi(t,\cdot) : \mathcal{H}\rightarrow \mathbb{R} \cup\left\{+\infty\right\}$ is a proper lower semi-continuous convex function. Then $\Psi$ is said to be \emph{twice epi-differentiable} at $x\in \Psi^{-1}(\mathord{\cdot} , \mathbb{R})$ for~$y\in\partial \Psi(0,\mathord{\cdot} )(x)$ if the family of second-order difference quotient functions $(\Delta_{t}^{2}\Psi(x|y))_{t>0}$ defined by
$$
  \fonction{\Delta_{t}^{2}\Psi(x|y) }{\mathcal{H}}{\mathbb{R}\cup\left\{+\infty\right\}}{z}{ \Delta_{t}^{2}\Psi(x|y) (z) := \displaystyle\frac{\Psi(t,x+t z)-\Psi(t,x)-t\dual{ y}{z}_{\mathcal{H}}}{t^{2}},}
$$
for all $t>0$, is Mosco epi-convergent. In that case, we denote by
$$
\mathrm{D}_{e}^{2}\Psi(x|y):=\mathrm{ME}\text{-}\mathrm{lim}~\Delta_{t}^{2}\Psi(x|y) ,
$$
which is called the \emph{second-order epi-derivative} of $\Psi$ at $x$ for $y$.
\end{myDefn}

The next proposition (which can be found in~\cite[Theorem~4.15]{8AB}) is the key point to derive our main results in the present work.

\begin{myProp}\label{TheoABC2018}
Let~$\Psi  :  \mathbb{R}_{+}\times \mathcal{H}\rightarrow \mathbb{R} \cup \left\{+\infty\right\}$ be a function such that, for all $t\geq0$, $\Psi(t,\cdot) : \mathcal{H}\rightarrow \mathbb{R} \cup\left\{+\infty\right\}$ is a proper, lower semi-continuous and convex function. Let $F  :  \mathbb{R}_{+}\rightarrow \mathcal{H}$ and~$u  :  \mathbb{R}_{+}\rightarrow \mathcal{H}$ be defined by
$$
    u(t):=\mathrm{prox}_{\Psi(t,\mathord{\cdot} )}(F(t)),
$$
for all~$t\geq 0$. If the conditions 
\begin{enumerate}
    \item[{\rm (i)}]  $F$ is differentiable at $t=0$;
    \item[{\rm (ii)}]  $\Psi$ is twice epi-differentiable at $u(0)$ for $F(0)-u(0)\in\partial \Psi(0,\mathord{\cdot} )(u(0))$;
    \item[{\rm (iii)}]  $\mathrm{D}_{e}^{2}\Psi(u(0)|F(0)-u(0))$ is a proper function on $\mathcal{H}$; 
\end{enumerate}
are satisfied, then $u$ is differentiable at $t=0$ with
$$
u'(0)=\mathrm{prox}_{\mathrm{D}_{e}^{2}\Psi(u(0)|F(0)-u(0))}(F'(0)).
$$
\end{myProp}

\section{Reminders on differential geometry}\label{appendix2}

In this appendix, let $\Omega$ be a nonempty bounded connected open subset of~$\R^{2}$ with a Lipschitz boundary $\Gamma :=\partial{\Omega}$ and $\nn$ be the outward-pointing unit normal vector to $\Gamma$. The next proposition, known as \textit{divergence formula}, can be found in~\cite[Theorem 4.4.7 p.104]{ALLNUM}. The following two propositions are also useful in the present paper and their proofs can be found in~\cite{HENROT}.

\begin{myProp}[Divergence formula]\label{div}
Consider the space
$$\HH_{\mathrm{div}}(\Omega, \R^{2\times 2}):= \{ w\in \mathrm{L}^{2}(\Omega,\R^{2\times 2})  \mid \mathrm{div} (w)\in \mathrm{L}^{2}(\Omega,\R^2) \},
$$
where $\mathrm{div}(w)$ is the vector whose the $i$-th component is defined by $\mathrm{div}(w)_i:=\mathrm{div}(w_{i})\in\LL^2(\Omega,\R)$, and where $w_i\in\mathrm{L}^{2}(\Omega,\R^{2})$ is the transpose of the~$i$-th line of~$w$, for all $i\in \{ 1,2 \}$. If $w\in \HH_{\mathrm{div}}(\Omega, \R^{2\times 2})$,
then $w$ admits a normal trace, denoted by~$w\nn \in \HH^{-1/2}(\Gamma,\R^2)$, satisfying
$$
\displaystyle\int_{\Omega}\mathrm{div}(w)\cdot v+\int_{\Omega}w:\nabla v=\dual{w\nn}{v}_{\HH^{-1/2}(\Gamma,\R^2)\times \HH^{1/2}(\Gamma,\R^2)}, \qquad\forall v \in \HH^1(\Omega,\R^{2}).
$$
\end{myProp}

\begin{myProp}\label{intbord}
Assume that~$\Gamma$ is of class $\mathcal{C}^2$ and let $\theta \in \mathcal{C}^{1}(\R^{2},\R^{2})$. It holds that
\begin{equation*}
    \int_{\Gamma}(\theta\cdot\nabla{v}+v\mathrm{div}_{\tau}(\theta))=\int_{\Gamma}\theta\cdot\nn(\partial_{\nn}v+Hv), \qquad\forall v \in \mathrm{W}^{2,1}(\Omega,\R),
\end{equation*}
where $\mathrm{div}_{\tau}(\theta):=\mathrm{div}(\theta)-(\nabla{ \theta}\nn \cdot \nn) \in \LL^\infty(\Gamma)$ is the {tangential divergence} of $\theta$,~$\partial_{\nn} v := \nabla v \cdot \nn \in \LL^1(\Gamma,\R)$ stands for the normal derivative of~$v$, and $H$ stands for the \textit{mean curvature} of $\Gamma$.
\end{myProp}

\begin{myProp}\label{beltrami}
Assume that $\Gamma$ is of class $\mathcal{C}^2$ and let $w\in\mathrm{H}^{2}(\Omega,\R^{2\times 2})$. It holds that 
\begin{equation*}
    \mathrm{div}(w)=\mathrm{div}_{\tau}\left(w_{\tau}\right)+H w\nn+\left(\partial_{\nn}w\right)\nn \qquad \text{\textit{a.e.}\ on } \Gamma,
\end{equation*}
where $\mathrm{div}_{\tau}\left(w_{\tau}\right)\in\LL^2(\Gamma,\R^2)$ is the vector whose the $i$-th component is defined by $\mathrm{div}_{\tau}\left(w_{\tau}\right)_i:=\mathrm{div}_{\tau}((w_{i})_{\tau})\in\LL^2(\Gamma,\R)$, where $(w_{i})_{\tau}:=w_i-(w_i\cdot\nn) \nn\in\LL^2(\Gamma,\R^2)$, and where~$\partial_{\nn}w\in\LL^2(\Gamma,\R^{2\times 2})$ is the matrix whose the~$i$-th line is the transpose of the vector $\partial_{\nn}w_{i}:=(\nabla{w_{i}})\nn\in\LL^2(\Gamma,\R^2)$, for all $i\in \{ 1,2 \}$. Moreover, it holds that
\begin{equation*}
\int_{\Gamma}v\cdot\mathrm{div}_{\tau}\left(w_{\tau}\right)=-\int_{\Gamma}w:\nabla_{\tau}v, \qquad \forall v \in \HH^{2}(\Omega,\R^2),
\end{equation*}
where $\nabla_{\tau}v$ is the matrix whose the $i$-th line is the transpose of the \textit{tangential gradient} $\nabla_{\tau}v_i:=\nabla{v_i}-(\partial_{\nn}v_i)\nn \in \HH^{1/2}(\Gamma,\R^{2})$, for all $i\in \{ 1,2 \}$.
\end{myProp}

\bibliographystyle{abbrv}
\bibliography{biblio}

\end{document}